\documentclass[letterpaper, reqno,11pt]{article}
\usepackage[margin=1.0in]{geometry}
\usepackage{geometry}                
\usepackage{graphicx}
\usepackage{amssymb}
\usepackage{epstopdf}
\usepackage{bm}

\usepackage{color} 

\usepackage{amsmath}
\usepackage{amsthm} 

\newtheorem{prop}{Proposition}[section]
\newtheorem{defn}{Definition}[section]
\newtheorem{conj}{Conjecture}[section]
\newtheorem{lemma}{Lemma}[section]
\newtheorem{cor}{Corollary}[section]
\newtheorem{theorem}{Theorem}[section]
\newtheorem{remark}{Remark}[section]

\newcommand{\eps}{\epsilon}

\newcommand{\ZZ}{\mathbb{Z}}
\newcommand{\RR}{\mathbb{R}}

\newcommand{\QQQ}{\mathcal{Q}}

\newcommand{\tubes}{\mathbb{T}}
\newcommand{\dist}{\operatorname{dist}}
\newcommand{\ofine}{\operatorname{fine}}
\newcommand{\ocoarse}{\operatorname{coarse}}

\newcommand{\gain}{1/40}
\newcommand{\gainmm}{\frac{1}{40}}
\title{Polynomial Wolff axioms and Kakeya-type estimates in $\RR^4$ }

\author{
Larry Guth\thanks{Massachusetts Institute of Technology, Cambridge, MA,  {\sl lguth@math.mit.edu}.}
\and
Joshua Zahl\thanks{University of British Columbia, Vancouver, BC,  {\sl jzahl@math.ubc.ca}.}
}

\begin{document}
\maketitle
\begin{abstract}
We establish new linear and trilinear bounds for collections of tubes in $\mathbb{R}^4$ that satisfy the polynomial Wolff axioms. In brief, a collection of $\delta$--tubes satisfies the Wolff axioms if not too many tubes can be contained in the $\delta$--neighborhood of a plane. A collection of tubes satisfies the polynomial Wolff axioms if not too many tubes can be contained in the $\delta$--neighborhood of a low degree algebraic variety. 

First, we prove that if a set of $\delta^{-3}$ tubes in $\mathbb{R}^4$ satisfies the polynomial Wolff axioms, then the union of the tubes must have volume at least $\delta^{1-\gain}$. We also prove a more technical statement which is analogous to a maximal function estimate at dimension $3+\gain$.  Second, we prove that if a collection of $\delta^{-3}$ tubes in $\mathbb{R}^4$ satisfies the polynomial Wolff axioms, and if most triples of intersecting tubes point in three linearly independent directions, then the union of the tubes must have volume at least $\delta^{3/4}$. Again, we also prove a slightly more technical statement which is analogous to a maximal function estimate at dimension $3+1/4$. 

We conjecture that every Kakeya set satisfies the polynomial Wolff axioms, but we are unable to prove this. If our conjecture is correct, it implies a Kakeya maximal function estimate at dimension $3+\gain$, and in particular this implies that every Kakeya set in $\RR^4$ must have Hausdorff dimension at least $3+\gain$. This would be an improvement over the current best bound of 3, which was established by Wolff in 1995.
\end{abstract}

\section{Introduction}
A Kakeya set in $\RR^n$ is a compact subset of $\RR^n$ that contains a unit line segment pointing in every direction. The Kakeya conjecture asserts that every Kakeya set in $n$ dimensions must have Hausdorff dimension $n$. In $\RR^2$, the conjecture was solved by Davies \cite{Davies}. In three and higher dimensions the conjecture remains open, though there has been partial progress. See \cite{W2, KT} for a survey of progress on the Kakeya problem. 

A key step when proving Kakeya estimates is to first discretize the problem. After this step has been performed, the Kakeya set is replaced by a finite set of ``$\delta$--tubes'' ($\delta$ neighborhoods of line segments), which point in $\delta$--separated directions; here $\delta>0$ is a small parameter. The Kakeya problem is transformed into the question of estimating the volume of the union of these tubes (or more precisely, the union of certain subsets of these tubes, which are known as ``shadings'' of the tubes). 

In \cite{W}, Wolff proved that every Kakeya set and every Nikodym set (a closely related object) in $\RR^n$ must have Hausdorff dimension at least $\frac{n+2}{2}$. To handle both Kakeya sets and Nikodym sets simultaneously, Wolff considered a more general type of set that satisfied the ``Wolff axioms''; both Kakeya and Nikodym sets satisfy these axioms. A set of $\delta$--tubes is said to satisfy the Wolff axioms if the cardinality of the set of $\delta$ tubes is $\delta^{1-n}$, and at most $t/\delta$ tubes can be contained in the intersection of the $t$--neighborhood of a line with the $\delta$-neighborhood of a plane. Every set of $\delta$ tubes that point in $\delta$-separated directions obeys the Wolff axioms.

In three dimensions, it is conjectured that the union of any set of tubes satisfying the Wolff axioms must have volume close to 1 (for reference, if the tubes were disjoint, then their union would have volume roughly 1).  This is a deep conjecture which would imply the Kakeya conjecture in $\RR^3$. In four and higher dimensions, however, the Wolff axioms are not sufficient to force the total volume to be close to 1.  For instance, Wolff's bound asserts that the union of any set of tubes satisfying the Wolff axioms in $\RR^4$ must have volume at least $\delta$, and this is in fact best possible---the set of tubes lying near a quadric hypersurface in $\RR^4$ satisy the Wolff axioms, but the union of these tubes has volume $\delta$.  The above example suggests that in four and higher dimensions, Wolff's axioms should be extended to not only forbid many tubes from lying near a plane, but to also forbid many tubes from lying near a low degree algebraic variety.  We define a polynomial version of the Wolff axioms in this spirit.  The precise definition is given in Definition \ref{genWolffDefn} below.

We conjecture that if a set of $\delta$-tubes in $\RR^n$ obeys the polynomial Wolff axioms, then the union of the tubes has volume close to 1.  In this paper, we study the four-dimensional case, and we prove that a set of tubes obeying the polynomial Wolff axioms obeys stronger estimates than a set of tubes that only obeys the regular Wolff axioms.  We prove that the union of any set of $\delta^{-3}$ tubes in $\RR^4$ that satisfy the polynomial Wolff axioms must have volume at least $\delta^{1-\gain+\eps}$. A key ingredient is a new trilinear Kakeya-type bound in $\RR^4$. We believe this bound may be of independent interest. To establish the trilinear bound, we use a ``grains decomposition'' lemma for Kakeya-type sets in $\RR^n$, which is related to the grains decomposition from \cite{Guth}. This lemma says that if the union of a set of $\delta$--tubes in $\RR^n$ has small volume, then this arrangement of tubes must have algebraic structure. More precisely, the union of tubes can be covered by the $\delta$--neighborhoods of pieces of algebraic varieties, which are called ``grains.''

To state our results precisely, we will first need several definitions. Throughout the paper, we will assume that all points, sets, etc.~are contained in the ball centered at the origin of radius 2. 

\begin{defn}
A \emph{$\delta$--tube} is the $\delta$--neighborhood of a unit line segment (recall that $\delta$--tubes, like all objects in this proof, must be contained in $B(0,2)$). \end{defn}

\begin{defn}
A \emph{semi-algebraic set} is a set of the form 
\begin{equation}\label{semiAlgDef}
S=\{x\in\RR^n\colon P_1(x)=0,\ldots,P_{k}(x)=0,\ Q_1(x)>0,\ldots,Q_{\ell}(x)>0\},
\end{equation}
 where $P_1,\ldots,P_k,Q_1,\ldots,Q_\ell$ are polynomials. We define the complexity of $S$ to be  $\min \big(\deg(P_1)+\ldots+\deg(P_k)+\deg(Q_1)+\ldots+\deg(Q_\ell)\big)$, where the minimum is taken over all representations of $S$ of the form \eqref{semiAlgDef}
\end{defn}
\begin{defn}\label{genWolffDefn}
Let $\tubes$ be a set of $\delta$--tubes. We say that $\tubes$ satisfies the \emph{polynomial Wolff axioms} if for every semi-algebraic set $S\subset\RR^n$ of complexity at most $E$, and every $\delta\leq\lambda\leq 1$, 
\begin{equation}\label{generalizedWolffIneq}
\big|\big\{T\in\tubes:\ |T\cap S|\geq\lambda|T| \big\}\big| \leq K_E |S| \delta^{1-n}\lambda^{-n}.
\end{equation}
Here $|\{\ldots\}|$ denotes the numbers of tubes (i.e.~counting measure), while $|S|$ denotes the Lebesgue measure of $S$.
\end{defn}

Note that \eqref{generalizedWolffIneq} is only meaningful if $S$ has positive Lebesgue measure, i.e.~if the semi-algebraic set $S$ has dimension $n$. The following remark should give some intuition for what it means to satisfy the polynomial Wolff axioms. 

\begin{remark}\label{implicationsOFWolffAxioms}
Let $\tubes$ be a set of $\delta$--tubes in $\RR^n$ that satisfy the polynomial Wolff axioms. Then the following properties hold:
\begin{itemize}
\item The tubes in $\tubes$ are essentially distinct: If $T\in\tubes$, then at most $C$ tubes from $\tubes$ are contained in the $10\delta$ neighborhood of $T$ (denoted $N_{10\delta}(T)$), where $C$ depends only on the constants $K_E,\ E=1,\ldots,6$ from Definition \ref{genWolffDefn}. The reason that $C$ only depends on $K_E$ for $E=1,\ldots,6$ is that the set $N_{10\delta}(T)$ is a semi-algebraic set of complexity $\leq 6$. Often, we will not worry about the exact complexity of the semi-algebraic sets we encounter, so we will replace the number 6 by $O(1)$. 
\item More generally, at most $Ct_1\cdots t_{n-1}\delta^{1-n}$ tubes are contained in a rectangular prism of dimensions $1\times t_1\times\ldots\times t_{n-1}$, where $C$ depends only on the constants $K_E,\ E=1,\ldots,O(1)$ from Definition \ref{genWolffDefn}. Any set of tubes $\tubes$ with the property that at most $Ct_1\cdots t_{n-1}\delta^{1-n}$ tubes are contained in a rectangular prism of dimensions $1\times t_1\times\ldots\times t_{n-1}$ is said to satisfy the \emph{linear Wolff axioms}\footnote{Note that Wolff's original axioms only required that at most $t/\delta$ tubes are contained in a rectangular prism of dimensions $1\times t_1\times \delta\times\ldots\times\delta$. Thus the linear Wolff axioms are more restrictive than the original Wolff axioms. Unlike the original Wolff axioms, however, the linear Wolff axioms are preserved by re-scalings.}.
\item At most $C\delta^{2-n}$ tubes are contained in the $M\delta$--neighborhood of an algebraic hypersurface $Z(P)$, where $C$ depends only on $M$ and the constant $K_{E}$ from \eqref{generalizedWolffIneq} with $E=2\deg P$.
\item More generally, if $B$ is a ball of radius $r$ then at most $C(\delta/r)\delta^{2-n}$ tubes satisfy $T\cap B\subset N_{C\delta}(Z(P))$.
\item If $Z$ is a $\ell$-dimensional algebraic variety, then at most $C\delta^{1-\ell}$ tubes are contained in the $M\delta$ neighborhood of $Z$, where again $C$ depends only on $M$ and and the constants $K_E$ from \eqref{generalizedWolffIneq}, with $E=1,\ldots,O_{\deg(Z)}(1)$. 
\item More generally, if $B$ is a ball of radius $r$ and $Z$ is an $\ell$-dimensional  variety, then at most $C(\delta/r)^{\ell-n}\delta^{1-\ell}$ tubes satisfy $T\cap B\subset N_{M\delta}(Z)$.
\item Let $\delta<\rho\leq 1$. If we only consider those tubes lying in the $\rho$--neighborhood of a line segment (we will call this set a cylinder), and if we re-scale this cylinder to have dimensions $1\times 1\times\ldots\times 1$, then the re-scaled tubes will satisfy the polynomial Wolff axioms at scale $\delta/\rho$ $\phantom{}$\footnote{Technically this is a lie, since the re-scaled tubes won't actually be $\delta/\rho$ tubes. However, the rescaled tubes are contained in $C\delta/\rho$ tubes, so the issue is easily fixable.}. Furthermore, if $\{K_E^\prime\}$ are the constants for the (rescaled) set of tubes, then $K_E^\prime\lesssim K_E$ for each $E$, where the implicit constant is independent of $E$. 
\end{itemize}
\end{remark}

\begin{defn}\label{lessapproxDefn}
We say $A\lesssim B$ if $A\leq CB$. Here and throughout the paper, $C$ will denote a constant (independent of $\delta$) that is allowed to change from line to line. Numbered constants $C_0,C_1,$ etc.~will have specific meanings, and won't be allowed to change. While $C$ can be any constant, we will think of it as being large. We will use $c$ to denote a small (positive) constant, which is also allowed to change from line to line.

We say $A\lessapprox B$ if for every $\epsilon>0$, there exists a constant $C_{\epsilon}$ (independent of $\delta$) so that $A\leq C_\epsilon\delta^{-\epsilon}B$. 
\end{defn}

\begin{conj}\label{polyWolffConj} (Kakeya conjecture for the polynomial Wolff axioms) For every dimension $n$, there is a complexity $E$ so that the following holds.  If $\tubes$ is a set of $\delta$-tubes in $\RR^n$ obeying the polynomial Wolff axioms for semi-algebraic sets of complexity at most $E$, then

$$ \Big\| \sum_{T \in \tubes} \chi_T \Big\|_{p} \lessapprox 1, \textrm{ for } p = \frac{n}{n-1}. $$

This would imply that

$$ \Big| \bigcup_{T \in \tubes} T \Big| \gtrapprox 1. $$

\end{conj}

Before this paper, the only known result about the polynomial Wolff axioms is Wolff's original result from \cite{W}:
if $\tubes$ obeys the original Wolff axioms then

$$ \Big\| \sum_{T \in \tubes} \chi_T \Big\|_{p} \lessapprox 1,\quad \textrm{ for } p = \frac{n + 2}{n}.$$

This maximal function bound implies that Kakeya and Nikodym sets have Hausdorff dimension at least $\frac{n+2}{2}$.  In this paper, we prove some stronger estimates in the four-dimensional case.

\subsection{Linear Kakeya-type bounds in $\RR^4$}

Our first result is a maximal function estimate for sets of tubes that satisfy the polynomial Wolff axioms.

\begin{theorem}\label{mainThm}
Let $\tubes$ be a set of $\delta$--tubes in $\RR^4$ that satisfy the polynomial Wolff axioms. Then
\begin{equation}\label{maximalFnBd}
\Big\Vert\sum_{T\in\tubes}\chi_T\Big\Vert_{d/(d-1)} \lessapprox \big(\frac{1}{\delta}\big)^{n/d - 1},\quad d = 3 + \gainmm.
\end{equation}
 \end{theorem}
See Proposition \ref{messyMainProp} for a slightly messier and more technical version of Theorem \ref{mainThm} that describes the implicit constant in \eqref{maximalFnBd} in greater detail. In particular, Proposition \ref{messyMainProp} explains how the implicit constant in \eqref{maximalFnBd} depends on the constants $\{K_E\}$ appearing in Definition \ref{genWolffDefn}.

Theorem \ref{mainThm} should be thought of as a maximal function bound of dimension $3+\gain$. In particular, Theorem \ref{mainThm} gives us a lower bound on the volume of unions of tubes satisfying the polynomial Wolff axioms.

\medskip

\noindent \emph{Added April 9, 2019:} In a previous version of this manuscript, a variant of Theorem \ref{mainThm} was claimed with exponent $d=3+1/28$ instead of $d=3+\gain$. The previous version of the manuscript contained an error (an exponent was not propagated properly from \eqref{WolffBoundOnTubes} to \eqref{pointwiseMuBdInsideTube}). This has now been corrected. 

\begin{cor}
Let $\tubes$ be a set of $\delta^{-3}$ $\delta$--tubes in $\RR^4$ that satisfy the polynomial Wolff axioms. Then
$$
\Big|\bigcup_{T\in\tubes} T\Big|\gtrapprox \delta^{1-\gain}.
$$
\end{cor}

Theorem \ref{mainThm} does not tell us anything about the Kakeya conjecture, because we do not know whether every direction-separated set of tubes satisfies the polynomial Wolff axioms. However, we conjecture that this should be the case.

\begin{conj}\label{KakeyaWolffConj}
Every set of tubes pointing in $\delta$--separated directions satisfies the polynomial Wolff axioms. More precisely, if $\tubes$ is a set of tubes pointing in $\delta$--separated directions, then $\tubes$ satisfies \eqref{generalizedWolffIneq} with constants $\{K_E\}$ that are independent of $\delta$.
\end{conj}
If Conjecture \ref{KakeyaWolffConj} is true, then Theorem \ref{mainThm} would imply a Kakeya maximal function estimate at dimension $3+\gain$. In particular, it would mean that every Kakeya set in $\RR^4$ must have Hausdorff dimension at least $3+\gain$. This would be a slight improvement over the previous best bound of 3 due to Wolff \cite{W}, and the related bound of \L{}aba-Tao \cite{LT} that every Kakeya set in $\RR^4$ must have upper Minkowski dimension at least $3+\epsilon_0$, where $\epsilon_0>0$ is a small absolute constant.

\begin{remark} 
Theorem \ref{mainThm} does not actually require the full strength of the polynomial Wolff axioms. The exact conditions needed for Theorem \ref{mainThm}  will be discussed further in Section \ref{minimalConditions} below.
\end{remark}

\subsection{Trilinear Kakeya-type bounds in $\RR^4$}
In \cite{BCT}, Bennett, Carbery and Tao proved that if $\tubes_1,\ldots,\tubes_n$ are sets of $\delta$--tubes in $\RR^n$, and if each tube in $\tubes_j$ makes a small angle with the $e_j$ direction, then for all $q>\frac{n}{n-1}$,
\begin{equation}\label{multiLinKakeya}
\bigg\Vert\prod_{j=1}^n \Big(\sum_{T\in\tubes_J}\chi_T\Big)\bigg\Vert_{L^{q/n}(\RR^n)}\lesssim \prod_{j=1}^n \big(\delta^{n/q}|\tubes_j|\big).
\end{equation}
The endpoint $q=\frac{n}{n-1}$ was later established by the first author in \cite{Guth2}. 

Heuristically, \eqref{multiLinKakeya} says that if $\tubes$ is a set of $\delta^{1-n}$ $\delta$-tubes in $\RR^n$, and if most $n$-tuples of intersecting tubes point in $n$ (quantitatively) linearly independent directions, then $\bigcup_{T\in\tubes}T$ has volume close to 1. In particular, if $\tubes$ is a collection of $\delta^{1-n}$ tubes for which $\bigcup_{T\in\tubes}T$ has small volume, then most of the tubes passing through a typical point in $\bigcup_{T\in\tubes}T$ must lie close to a hyperplane. 

Inequality \eqref{multiLinKakeya} deals with the situation where the number of families of tubes is the same as the dimension of the ambient Euclidean space (\eqref{multiLinKakeya} is an inequality in $\RR^n$, and there are $n$ families of tubes). However, it can be easily extended to the case where there are $\ell\leq n$ families of tubes in $\RR^n$. Heuristically, this says that if $\tubes$ is a set of $\delta^{1-n}$ $\delta$-tubes in $\RR^n$, and if most $\ell$-tuples of intersecting tubes point in $\ell$ quantitatively linearly independent directions, then $\bigcup_{T\in\tubes}T$ has volume at least $\delta^{n-\ell}$. This result is sharp: if $\delta^{1-n}$ $\delta$ tubes are placed at random into the $\delta$--neighborhood of an $\ell$-dimensional flat in $\RR^n$, then the union of these tubes has volume $\leq\delta^{n-\ell}$, and most $\ell$--tuples of intersecting tubes will point in $\ell$ (quantitatively) linearly independent directions.

Theorem \ref{3linKakeyaR4} below is a stronger inequality for three collections of tubes $\tubes_1,\tubes_2,\tubes_3$ in $\RR^4,$ if the tubes in each of $\tubes_1,\tubes_2,$ and $\tubes_3$ satisfy the polynomial Wolff axioms. Heuristically, Theorem \ref{3linKakeyaR4} says that if $\tubes$ is a set of $\delta^{-3}$ $\delta$-tubes in $\RR^4$ that satisfy the polynomial Wolff axioms, and if most triples of intersecting tubes point in three quantitatively linearly independent directions, then $\bigcup_{T\in\tubes}T$ has volume at least $\delta^{3/4}$. 

\begin{theorem}\label{3linKakeyaR4}
Let $\tubes$ be a set of $\delta$-tubes in $\RR^4$ that satisfy the polynomial Wolff axioms. Then
\begin{equation}\label{trilinearBoundFixedP}
\int\Big(\sum_{T_1,T_2,T_3\in\tubes}\chi_{T_1}\ \chi_{T_2}\ \chi_{T_3}\ |v_1\wedge v_2\wedge v_3|^{12/13}\Big)^{13/27} \lessapprox \delta^{-1/3},
\end{equation}
where in the above expression $v_i$ is the direction of the tube $T_i$.
\end{theorem}

See Proposition \ref{techVersion3lin} for a slightly messier and more technical version of Theorem \ref{mainThm} that describes the implicit constant in \eqref{trilinearBoundFixedP} in greater detail. In particular, Proposition \ref{techVersion3lin} explains how the constant depends on the constants $\{K_E\}$ appearing in Definition \ref{genWolffDefn}.

\section{Preliminaries}
\subsection{Shadings, the two-ends condition, and dyadic pigeonholing}
\begin{defn}[Two-ends condition]
Let $T\subset\RR^n$ be a $\delta$--tube. We call a set $Y(T)\subset T$ a \emph{shading} of $T$. We will often use the variable $\lambda$ to denote the quantity $|Y(T)|/|T|$. If $Y(T)$ satisfies the bound
\begin{equation}
|Y(T)\cap B(x,r)|\leq \alpha r^{\epsilon_0}|Y(T)|
\end{equation}
for all $x\in B(0,2)$ and all $\delta\leq r\leq 1$, then we say $Y(T)$ satisfies the \emph{two-ends condition} with exponent $\epsilon_0$ and error $\alpha$.
\end{defn}

The following lemma says that if $Y(T)$ is a shading of $T$, then we can always find a large subset of $Y(T)$ that satisfies a (re-scaled) two-ends condition. The lemma below was first used by Wolff in \cite{W}. A proof of the lemma as stated here can also be found in \cite[Lemma 6]{T2}.
\begin{lemma}\label{twoEndsLem}
Let $T$ be a tube and let $Y(T)\subset T$ be a shading with $|Y(T)|\geq\delta$. Let $0<\eps_0<1$. Then there is a ball $B(x,r)$ with $\delta\leq r\leq 1$ and 
$$
|Y(T) \cap B(x,r)|\geq\delta^{\eps_0}|Y(T)|
$$ 
so that for all $x^\prime\in \RR^3$ and all $\delta\leq r^\prime\leq 1$ we have
\begin{equation}\label{consequenceOfTwoEnds}
|Y(T)\cap B(x^\prime, r^\prime)\cap B(x,r)|\leq (r^\prime/r)^{\eps_0}|Y(T)\cap B(x,r)|.
\end{equation}
\end{lemma}

\begin{defn}
Let $\tubes$ be a set of $\delta$--tubes, and for each $T\in\tubes$ let $Y(T)$ be a shading of $T$. A \emph{refinement} of $Y$ is a set $\tubes^\prime\subset \tubes$, and for each $T\in\tubes^\prime$ a set $Y^\prime(T)\subset Y(T)$ so that $\sum_{T\in\tubes^\prime}|Y^\prime(T)|\gtrsim |\log\delta|^{-C}\sum_{T\in\tubes}|Y(T)|$, where $C$ is an absolute constant (to be pedantic, we should call this a $C$--refinement, but in practice the constant $C$ will always be at most 3). We will sometimes abuse notation and use the same symbols $\tubes, Y$ to denote the refinement of $\tubes$ and $Y$.
\end{defn}

For example, if $(\tubes,Y)$ is a set of tubes and their associated shadings with $|\tubes|\leq\delta^{-C}$, then the function $\sum_{T\in\tubes}\chi_{Y(T)}(x)$ can take integer values between $0$ and $|\tubes|$. However, there exists a set $B$, a number $\mu$, and a refinement $\tubes^\prime=\tubes,\ Y^\prime(T)\subset Y(T)$ so that $\sum_{T\in\tubes^\prime}\chi_{Y^\prime(T)}(x)\sim\mu\chi_B$ pointwise. 

Similarly, if $(\tubes, Y)$ is a set of tubes and their associated shadings with $\delta^{C}\leq Y(T)\leq |T|$ for each $T\in\tubes$, then there exists a number $\lambda$, a refinement $\tubes^\prime\subset \tubes$ and $Y^\prime(T)=Y(T)$ so that $\lambda\leq|Y^\prime(T)|/|T|\leq2\lambda$ for all $T\in\tubes^\prime$. These two types of arguments will occur frequently in the proof, and they are the only places where refinements will be used.

\subsection{The two-ends reduction}

First, we will state a slightly more precise (and uglier) version of Theorem \ref{mainThm} that explicitly describes the different constants involved in the bound.

\begin{prop}[Messy version of Theorem \ref{mainThm}] \label{messyMainProp}
For all $\epsilon>0$, there exist constants $c_\epsilon>0$ and $d(\epsilon)$ so that the following holds. Let $\tubes$ be a set of $\delta$--tubes in $\RR^4$ that satisfy the polynomial Wolff axioms. For each $T\in\tubes$, let $Y(T)\subset T$ with $\lambda\leq |Y(T)|/|T|\leq 2\lambda$. Then 
\begin{equation} \label{volumeBoundExplicit}
\Big| \bigcup_{T\in\tubes}Y(T)\Big| \geq c_{\epsilon} \lambda^{3+\gain}K^{-1} \delta^{1-\gain+\epsilon}\big(\delta^3|\tubes|\big),
\end{equation}
where $K=K_{d(\epsilon),\tubes}=\sup_{1\leq E\leq d(\epsilon)}{K_E}$, and the constants $\{K_E\}$ are from Definition \ref{genWolffDefn}.
\end{prop}
\begin{remark}
Many statements throughout the proof will begin with the phrase ``for all $\epsilon>0$, there exist constants $c_\epsilon>0$ and $d(\epsilon)$ so that....'' The reader should think of the constant $c_{\epsilon}$ as differing between statements, but the function $d(\epsilon)$ is universal. \end{remark}
\begin{prop}[Two-ends version of Proposition \ref{messyMainProp}]\label{messyMainPropTwoEnds}
For all $\epsilon>0$, $\epsilon_0>0$ there exist constants $c_{\epsilon,\epsilon_0}>0, C^\prime_{\epsilon,\epsilon_0}$, and $d(\epsilon)$ so that the following holds. Let $\tubes$ be a set of $\delta$--tubes in $\RR^4$ that satisfy the polynomial Wolff axioms. For each $T\in\tubes$, let $Y(T)\subset T$ with $\lambda\leq |Y(T)|/|T|\leq 2\lambda$. Suppose that each tube $T\in\tubes$ satisfies the two-ends condition with exponent $\epsilon_0$ and error $\alpha$. Then 
\begin{equation} \label{volumeBoundExplicitTwoEnds}
\Big| \bigcup_{T\in\tubes}Y(T)\Big| \geq c_{\epsilon,\epsilon_0}\alpha^{-C^\prime_{\epsilon,\epsilon_0}} \lambda^{3+\gain}K^{-1} \delta^{1-\gain+\epsilon}\big(\delta^3|\tubes|\big),
\end{equation}
where $K=K_{d(\epsilon),\tubes}=\sup_{1\leq E\leq d(\epsilon)}{K_E}$, and the constants $\{K_E\}$ are from Definition \ref{genWolffDefn}.
\end{prop}
\begin{proof}[Proof of Proposition \ref{messyMainProp} using Proposition \ref{messyMainPropTwoEnds}]
The reduction from Proposition \ref{messyMainProp} to Proposition \ref{messyMainPropTwoEnds} is a standard application of the ``two-ends reduction'' argument. The one new feature is that instead of pointing in $\delta$--separated directions, the tubes in $\tubes$ instead satisfy the polynomial Wolff axioms. As a result, we must take greater care when performing the ``rescaling'' part of the argument.

Fix $\epsilon>0$. Let $\epsilon_0=\epsilon/(6+1/14)$. Apply Lemma \ref{twoEndsLem} to each tube $T\in\tubes$. Let $B(x,r)$ be the resulting ball, and define $Y^\prime(T) = Y(T)\cap B(x,r)$. Note that 
$$
|Y^\prime(T)|\geq r^{\epsilon_0}\lambda |T|,
$$ 
and 
$$
|B(x^\prime,r^\prime)\cap Y^\prime(T)|\leq (r^\prime/r)^{\epsilon_0} |Y^\prime(T)|
$$
for all $x^\prime\in\RR^4$ and all $r^\prime\leq r$.

Each tube $T\in\tubes$ has an associated value of $r_T$ from the two-ends reduction described above, with $\delta\leq r_T\leq 1$. For each tube $T\in\tubes,$ we also have $\delta^{\eps_0}\lambda \leq|Y^\prime(T)|\leq 2\lambda$. We will now dyadically pigeonhole the tubes in $\tubes$ based on this value of $r_T$ and $|Y^\prime(T)|$. We will select a value of $\delta\leq r_0\leq 1$ and $\delta^{\eps_0}\lambda\leq \lambda_0\leq\lambda$ and a set $\tubes^\prime\subset\tubes$ with $|\tubes^\prime|\geq |\log\delta|^2|\tubes|$ so that for each $T\in\tubes^\prime$, we have $r_T\sim r_0$, and $\lambda^\prime\leq |Y^\prime(T)|\leq 2\lambda^\prime$. 

Cover $B(0,2)$ by boundedly overlapping balls of radius $10r_0,$ so that each ball of radius $r_0$ is entirely contained within one of the balls. Associate each $T\in\tubes^\prime$ to one of the balls. This gives us a partition $\tubes^\prime=\bigsqcup \tubes^\prime_B$. 


For each ball $B$, re-scale the tubes in $\tubes^\prime_B$ to have dimensions $1\times \delta/r_0\times\delta/r_0\times\delta/r_0$, and denote this new set by $\tilde \tubes^\prime_B$. Observe that if $\{\tilde K_E\}$ are the Wolff constants associated to $\tilde \tubes^\prime_B$, and if $\{K_E\}$ are the Wolff constants associated to $\tubes$, then $\tilde K_E = r_0^{-3}K_E$ for each $E$ (note that this scaling factor of $r_0^{-3}$ is the same scaling factor one would expect for tubes pointing in $\delta$--separated directions). To see this, let $\tilde S\subset B(0,2)$ be a semi-algebraic set of complexity $E$, and let $S$ be the pre-image of $\tilde S$ under the re-scaling that sends $B$ to the ball $B(0,2)$. Then
\begin{equation}
\begin{split}
|\{\tilde T\in \tilde\tubes\colon |\tilde T\cap\tilde S|\geq\lambda_1|\tilde T| \}| &\leq |\{T\in \tubes\colon |T\cap S|\geq  r_0\lambda_1|T| \}|\\
& \leq K_E |S|\delta^{-3}(r_0\lambda_1)^{-4}\\
&=K_E |\tilde S|\delta^{-3} \lambda_1^{-4}\\
&=K_Er_0^{-3}|\tilde S|(\delta/r_0)^{-3}\lambda_1^{-4}.
\end{split}
\end{equation}
Let $\tilde Y(\tilde T)$ be the re-scaled version of $Y^\prime(T)$. Then $\tilde Y$ satisfies the two-ends condition with exponent $\eps_0$ and error $\alpha = 1$. Furthermore, $|\tilde Y(\tilde T)|\sim (\lambda^\prime/r_0)$ for each $\tilde T\in\tilde\tubes^\prime_B.$

Apply Proposition \ref{messyMainPropTwoEnds} to $\tilde\tubes^\prime_B$. We conclude that
\begin{equation*}
\begin{split}
\Big|\bigcup_{\tilde T\in \tilde\tubes^\prime_B}\tilde Y(\tilde T) \Big|&\geq c_{\epsilon/2, \epsilon_0} (\lambda^\prime/r_0)^{3+\gain}(r_0^{-3}K)^{-1} (\delta/r)^{1-\gain+\epsilon/2}\big((\delta/r_0)^3|\tilde\tubes^\prime_B|\big)\\&
= r_0^{-4} c_{\epsilon/2,\epsilon_0} (\lambda^\prime)^{3+\gain}K^{-1}\delta^{1-\gain+\epsilon/2}\big(\delta^3|\tilde\tubes^\prime_B|\big).
\end{split}
\end{equation*}
and thus after undoing the scaling $\delta\to\delta/r_0$, we have
\begin{equation}\label{unionTubesInBall}
\Big|\bigcup_{T\in\tubes\colon \tilde T\in \tubes^\prime_B}Y^\prime(T) \Big|\geq c_{\epsilon/2} (\lambda^\prime)^{3+\gain}K^{-1}\delta^{1-\gain+\epsilon/2}\big(\delta^3|\tilde\tubes^\prime_B|\big).
\end{equation}
Since $|\tubes|\geq\delta^{\epsilon/2}\sum_B|\tilde\tubes^\prime_B|$, and the sets in \eqref{unionTubesInBall} are at most $10^6$--fold overlapping for different balls $B$, we have
\begin{equation}
\begin{split}
\Big|\bigcup_{T\in \tubes}Y(T) \Big|&\geq c\sum_{B} \Big|\bigcup_{T\in\tubes\colon \tilde T\in \tubes^\prime_B}Y^\prime(T) \Big|\\
&\geq c \sum_B \geq c_{\epsilon/2,\epsilon_0} (\lambda^\prime)^{3+\gain}K^{-1}\delta^{1-\gain+\epsilon/2}\big(\delta^3|\tilde\tubes^\prime_B|\big)\\
&\geq c\ c_{\epsilon/2,\epsilon_0} \lambda^{3+\gain}\delta^{(3+\gain)\epsilon/(6+1/14)}K^{-1}\delta^{1-\gain+\epsilon/2}\big(\delta^3|\tubes|\big)\\
&\geq c\ c_{\epsilon/2,\epsilon_0} \lambda^{3+\gain}K^{-1}\delta^{1-\gain+\epsilon}\big(\delta^3|\tubes|\big).\qedhere
\end{split}
\end{equation}
\end{proof}
\subsection{The robust transversality reduction}\label{rescalingArgumentsSec}

Theorem \ref{3linKakeyaR4} gives the strongest bounds when most triples of intersecting tubes point in three linearly independent directions. As a starting point, one would hope that most pairs of intersecting tubes point in two linearly independent directions. In this section, we will show that we can always reduce to this situation.

\begin{defn}
Let $(\tubes,Y)$ be a set of $\delta$ tubes and their associated shadings. We say that $(\tubes,Y)$ is $s$--robustly transverse (with error $t$) if for all $x\in\RR^4$ and all vectors $v$, we have
\begin{equation}\label{quantTrans}
|\{T\in\tubes \colon x\in Y(T),\ \angle(T,v)< s \}|\leq t |\{T\in\tubes\colon x\in Y(T)\}|.
\end{equation}
\end{defn}

We wish to reduce Proposition \ref{messyMainPropTwoEnds} to the case where $\tubes$ is robustly transverse. This reduction will involve an induction argument. The next proposition is identical to Proposition \ref{messyMainPropTwoEnds}, except we have added the requirement that the tubes are robustly transverse.

\begin{prop}\label{mainPropQuantTrans}

For all $\epsilon>0$, $\epsilon_0>0$ there exist constants $c_{\epsilon,\epsilon_0}>0, C^\prime_{\epsilon,\epsilon_0}$, and $d(\epsilon)$ so that the following holds. Let $\tubes$ be a set of $\delta$--tubes in $\RR^4$ that satisfy the polynomial Wolff axioms. For each $T\in\tubes$, let $Y(T)\subset T$ with $\lambda\leq |Y(T)|/|T|\leq 2\lambda$. Suppose that $(\tubes,Y)$ is $s$--robustly transverse (with error $1/100$) and that each tube $T\in\tubes$ satisfies the two-ends condition with exponent $\epsilon_0$ and error $\alpha$. Then 
\begin{equation} \label{volumeBoundExplicitTrans}
\Big| \bigcup_{T\in\tubes}Y(T)\Big| \geq c_{s}c^\prime_{\epsilon,\epsilon_0}\alpha^{-C^\prime_{\epsilon,\epsilon_0}}\lambda^{3+\gain}K^{-1} \delta^{1-\gain+\epsilon}\big(\delta^3|\tubes|\big),
\end{equation}
where $K=K_{d(\epsilon),\tubes}=\sup_{1\leq E\leq d(\epsilon)}{K_E}$, and the constants $\{K_E\}$ are from Definition \ref{genWolffDefn}.
\end{prop}

We will show that Proposition \ref{mainPropQuantTrans} implies Proposition \ref{messyMainPropTwoEnds}. 

\begin{proof}[Proof of Proposition \ref{messyMainPropTwoEnds} using Proposition \ref{mainPropQuantTrans}]
Suppose that Proposition \ref{mainPropQuantTrans} holds; we will prove Proposition \ref{messyMainPropTwoEnds} by induction on $\delta$. Fix the value of $\epsilon$ and $\epsilon_0$ from the statement of Proposition \ref{messyMainPropTwoEnds}. We will assume that $\epsilon<1/4$. 

By making the constant $c_{\epsilon,\epsilon_0}$ sufficiently small, we can assume that Proposition \ref{messyMainPropTwoEnds} holds for all $\delta>0$ satisfying 
\begin{equation}\label{smallLog}
|\log\delta|\leq \delta^{-\epsilon/100}.
\end{equation}

Suppose Proposition \ref{messyMainPropTwoEnds} has been established for all $\delta^\prime>\delta$, and let $(\tubes,Y)$ be a set of $\delta$ tubes satisfying the hypotheses of Proposition \ref{messyMainPropTwoEnds}. Let $X\subset\bigcup_{T\in\tubes}Y(T)$ be the set where \eqref{quantTrans} holds with $s>0$ a small constant (depending only on $\epsilon$ and $\epsilon_0$) to be determined later. Either
\begin{enumerate}
\item[(A):] $\sum_{T\in\tubes}|Y(T)\cap X|\geq\frac{1}{2}\sum_{T\in\tubes}|Y(T)|$, or
\item[(B):] (A) fails
\end{enumerate}
Suppose (A) holds. Let $\tubes^\prime=\{T\in\tubes\colon |Y(T)\cap X|\geq \frac{1}{4}\lambda|T|\}$. Then $|\tubes^\prime|\geq \frac{1}{4}|\tubes|$. For each $T\in\tubes^\prime$, let $Y^\prime(T)\subset T$ with $\lambda/4\leq |Y^\prime(T)|/|T|\leq \lambda/2$. Each of the tubes still satisfies the two-ends condition with exponent $\epsilon_0$ and error $\alpha/4$. Apply Proposition \ref{mainPropQuantTrans} to $\tubes^\prime$ with the shading $Y^\prime(T)$ and with the value of $\epsilon$ and $\epsilon_0$ specified above. We conclude that
\begin{equation*}
\begin{split}
\Big| \bigcup_{T\in\tubes}Y(T)\Big| &\geq 
\Big| \bigcup_{T\in\tubes}Y^\prime(T)\Big| \\
&\geq c_s c^\prime_{\epsilon,\epsilon_0}(4\alpha)^{-C^\prime_{\epsilon,\epsilon_0}}(\frac{1}{4}\lambda)^{3+\gain}K^{-1}\delta^{1-\gain+\epsilon}(\delta^3|\tubes^\prime|)\\
&\geq 4^{-C^\prime_{\epsilon,\epsilon_0}-4} c_s c^\prime_{\epsilon,\epsilon_0}\alpha^{-C^\prime_{\epsilon,\epsilon_0}}\lambda^{3+\gain}K^{-1} \delta^{1-\gain+\epsilon}(\delta_3|\tubes|).
\end{split}
\end{equation*}
Thus Proposition \ref{messyMainPropTwoEnds} holds as long as $c_{\epsilon,\epsilon_0}\leq c^\prime_{\epsilon,\epsilon_0}c_s4^{-C^\prime_{\epsilon,\epsilon_0}-4}$. 

Now suppose (B) holds. Let $A=\RR\backslash X$. Then for every $x\in A$, there is a vector $v_x$ so that
\begin{equation}\label{popularRho}
|\{T\in\tubes\colon x\in Y(T),\ \angle(T, v_x)<s \}|\geq \frac{1}{100} |\{T\in\tubes\colon x\in Y(T) \}|.
\end{equation}

For each $T\in\tubes,$ let $Y^\prime(T)=Y(T)\cap \{x \in A \colon \angle(T,v_x)< s\}$. Then 
\begin{equation}\label{boundOnYPrime}
\sum_{T\in\tubes}|Y^\prime(T)|\geq \frac{1}{200}\lambda(\delta^3|\tubes|).
\end{equation}

Let $\tubes^\prime=\{T\in\tubes\colon |Y^\prime(T)|\geq \frac{1}{400}\lambda|T|\}$. Then $|\tubes^\prime|\geq \frac{1}{400}|\tubes|$. We can refine each shading $Y^\prime(T)$ slightly so that $\lambda/400\leq |Y^\prime(T)|/|T|\leq \lambda/200$. Each of the tubes still satisfies the two-ends condition with exponent $\epsilon_0$ and error $400\alpha$.

Cover the sphere $S^3$ with $\leq 100$-fold overlapping caps of radius $3 s$, so that every ball (in $S^3$) of radius $s$ is entirely contained in one of the caps. Note that if $x\in A$, then there is a cap $\tau$ so that every tube $T\in\tubes^\prime$ with $x\in Y^\prime(T)$ points in a direction lying in $\tau$. For each cap $\tau$, let $\tubes^\prime(\tau)\subset\tubes^\prime$ be a set of tubes pointing in directions lying in $\tau$, so that $\tubes^\prime=\bigsqcup_{\tau}\tubes^\prime(\tau)$ is a partition of $\tubes^\prime$. For each cap $\tau$, partition $B(0,2)$ into $\leq 100$--fold overlapping cylinders of dimensions $1\times 10\rho\times 10\rho\times 10\rho$; call this set of cylinders $\operatorname{Cyl}(\tau)$. Note that if $T\in\tubes^\prime(\tau)$, then $T$ is contained in at least one of these cylinders. For each such cylinder $U$, $\tubes^\prime(\tau,U)\subset\tubes^\prime(\tau)$, so that 
$$
\tubes^\prime=\bigsqcup_{\tau}\bigsqcup_{U\in \operatorname{Cyl}(\tau)}\tubes^\prime(\tau,U),
$$
i.e.~each tube from $\tubes$ is assigned to a cap $\tau$ and a $100\rho$ cylinder pointing in the direction $\tau.$  

Observe that the sets $\{\bigcup_{T\in\tubes(\tau,U)}Y^\prime(T) \}_{\tau, U}$ are at most $10^4$--fold overlapping. Thus
\begin{equation}\label{disjointness}
\Big|\bigcup_{T\in\tubes^\prime}Y(T) \Big|\geq 10^{-4}\sum_{\tau}\sum_{U\in \operatorname{Cyl}(\tau)}\Big|\bigcup_{T\in \tubes^\prime(\tau,U)}Y^\prime(T)\Big|.
\end{equation}

Now, for each cap $\tau$ and each cylinder $U\in \operatorname{Cyl}(\tau)$, let $L$ be a line pointing in the same direction as $U$ and distance $100 s$ from $U$. Let $f\colon\RR^4\to\RR^4$ be the map that fixes $L$ and dilates $\RR^4$ by a factor of $s^{-1}$ in all directions orthogonal to $L$. Then $f(U)$ contains a ball of radius $1/1000$ and is contained in a ball of radius $1000$; if $Y\subset U$ is a set, then 
\begin{equation}\label{dilationFactorOff}
\frac{1}{1000}\rho^{-3}|Y|\leq |f(Y)|\leq 1000\rho^{-3}|Y|;
\end{equation}
if $T\in\tubes(\tau,U)$, then $f(T)$ is contained in a $1000\rho^{-1}\delta$ tube and contains a $\frac{1}{1000}\rho^{-1}\tubes$. For each $T\in\tubes^\prime(\tau,U)$, let $\tilde T$ be a $1000\rho^{-1}\delta$--tube that contains $f(T)$, and let $\tilde Y(T)=f(Y(T)).$ Then $\tilde\tubes=\{\tilde T\colon T\in \tubes(\tau,U)\}$ satisfies the polynomial Wolff axioms (at scale $\rho^{-1}\delta$). By \eqref{dilationFactorOff}, the constant $K$ associated to $\tilde\tubes$ is at most 1000 times the constant $K$ associated to $\tubes^\prime$ (which is also the constant associated to $\tubes$) ($K$ also depends on $\epsilon$, but we have fixed a value of $\epsilon$ throughout this proof.)

Observe that the tubes in $\tilde\tubes$ have thickness $\delta^\prime=1000\rho^{-1}\delta<\delta$, and by the induction hypothesis, we know that Proposition \ref{messyMainPropTwoEnds} holds for this value of $\delta^\prime$. Thus we can apply Proposition \ref{messyMainPropTwoEnds} to $(\tilde\tubes,\tilde Y)$ (with the same value of $\epsilon$ and $\epsilon_0$ as above). We conclude that
$$
\Big|\bigcup_{\tilde T\in\tilde\tubes}\tilde Y(T)\Big|\geq  c_{\epsilon,\epsilon_0}(400\alpha)^{-C^\prime_{\epsilon,\epsilon_0}}(10^{-5} \lambda)^{3+\gain}(10^3K)^{-1}(\delta/s)^{1-\gain+\epsilon}\big( (\delta/s)^3|\tilde\tubes|\big), 
$$
and thus by \eqref{dilationFactorOff},
\begin{equation}\label{estimateInsideTubes}
\Big|\bigcup_{T\in\tubes(\tau,U) }Y^\prime(T)\Big|\geq c_{\epsilon,\epsilon_0}(400\alpha)^{-C^\prime_{\epsilon,\epsilon_0}}(10^{-5}\lambda)^{3+\gain}(10^3K)^{-1}(\delta/s)^{1-\gain+\epsilon}\big( \delta^3|\tubes(\tau,U)|\big).
\end{equation} 
Combining \eqref{disjointness} and \eqref{estimateInsideTubes}, we conclude
\begin{equation}\label{puttingItAllTogether}
\begin{split}
\Big|\bigcup_{T\in\tubes^\prime}Y(T) \Big|&\geq 10^{-20-3C^\prime_{\epsilon,\epsilon_0}}\sum_{\tau}\sum_{U\in \operatorname{Cyl}(\tau)}c_{\epsilon,\epsilon_0}\alpha^{-C^\prime_{\epsilon,\epsilon_0}}\lambda^{3+\gain}K^{-1}(\delta/s)^{1-\gain+\epsilon}\big( \delta^3|\tubes(\tau,U)|\big)\\
&\geq c_0\ c_{\epsilon,\epsilon_0}\alpha^{-C^\prime_{\epsilon,\epsilon_0}}\lambda^{3+\gain}K^{-1}(\delta/s)^{1-\gain+\epsilon}\big( \delta^3|\tubes|\big),
\end{split}
\end{equation}
Thus, provided we select $s$ sufficiently small (depending only on $C^\prime_{\epsilon,\epsilon_0},$ which in turn depends only on $\epsilon$ and $\epsilon_0$) so that $s^{1-\gain-\epsilon} < s^{1/2} < 10^{-20-3C^\prime_{\epsilon,\epsilon_0}}$, then 
$$
\Big|\bigcup_{T\in\tubes}Y(T) \Big|\geq c_{\epsilon,\epsilon_0}\alpha^{-C^\prime_{\epsilon,\epsilon_0}}\lambda^{3+\gain} K^{-1}\delta^{1-\gain+\epsilon}(\delta^3|\tubes|),
$$
which closes the induction and completes the proof of Proposition \ref{messyMainPropTwoEnds}. 
\end{proof}
In Section \ref{proofOfMainThmSection} we will prove Proposition \ref{mainPropQuantTrans}.
\section{Trilinear Kakeya in $\RR^4$}\label{trilinearKakeyaR4Sec}
In this section, we will prove Theorem \ref{3linKakeyaR4}. First, we will state a slightly more technical version of the theorem
\begin{prop}\label{techVersion3lin}
For all $\epsilon>0$, there exist constants $C_\epsilon, d(\epsilon)$ so that the following holds. Let $\tubes$ be a set of $\delta$ tubes in $\RR^4$ that satisfy the polynomial Wolff axioms. Then
\begin{equation}\label{trilinearBoundFixedP}
\int\Big(\sum_{T_1,T_2,T_3\in\tubes}\chi_{T_1}\ \chi_{T_2}\ \chi_{T_3}\ |v_1\wedge v_2\wedge v_3|^{12/13}\Big)^{13/27} \leq C_{\epsilon}\delta^{-1/3-\epsilon}K^{1/9}(\delta^3|\tubes|)^{4/3},
\end{equation}
where in the above expression $v_i$ is the direction of the tube $T_i$, and $K=K_{\tubes, d(\epsilon)}=\sup_{1\leq E\leq d(\epsilon)}{K_E}$, where $\{K_E\}$ are the constants from Definition \ref{genWolffDefn}
\end{prop}

\begin{cor}\label{volumeTrilinearCor}
For all $\epsilon>0$, there exist constants $c_\epsilon>0, d(\epsilon)$ so that the following holds. Let $\tubes$ be a set of $\delta$--tubes in $\RR^4$ that satisfy the polynomial Wolff axioms. For each $T\in\tubes,$ let $Y(T)\subset T$ with $|Y(T)|\sim\lambda|T|$. Suppose that $(\tubes,Y)$ is $s$--robustly transverse with error $1/100$, and that for all $x\in\RR^4$ and all 2--planes $\Pi$, we have 
\begin{equation}\label{notConcentratedInThetaPlane}
|\{T\in\tubes\colon x\in Y(T),\ \angle(T,\Pi)<\theta\}|\leq \frac{1}{100}|\{T\in\tubes\colon x\in Y(T)\}|.
\end{equation}
Then
\begin{equation}\label{volumeOfTrlinearTubes}
\Big|\bigcup Y(T)\Big|\geq c_\epsilon c_s \lambda^{3+1/4}K^{-1/4} \theta\delta^{3/4+\epsilon}(\delta^3|\tubes|)^{1/4},
\end{equation}
where $K=\sup_{1\leq E\leq d(\epsilon/9)}{K_E}$, and $\{K_E\}$ are the constant from Definition \ref{genWolffDefn}
\end{cor}

\begin{remark}
Heuristically, Corollary \ref{volumeTrilinearCor} says that if $\tubes$ is a set of $\delta^{-3}$ essentially distinct $\delta$--tubes that satisfy the polynomial Wolff axioms, and if most triples of tubes passing through a typical cube in $\RR^4$ span three quantitatively linearly independent directions, the the union of the tubes has volume at least $\delta^{3/4};$ this corresponds to a Hausdorff dimension bound of $3+1/4$.

In contrast, the multilinear Kakeya theorem says that the union of these tubes has volume at least $\delta$ (this is a weaker statement that corresponds to a Hausdorff dimension bound of 3).
\end{remark}

\begin{remark}
If $\tubes$ satisfies the polynomial Wolff axioms, then since each tube in $\tubes$ is contained in $B(0,2)$, and $B(0,2)$ is a semi-algebraic set of measure $\leq 100$ and complexity 2, we have $(\delta^3|\tubes|)\leq 100 C_2$, where $C_2$ is the constant from Definition \ref{genWolffDefn}. Thus \eqref{volumeOfTrlinearTubes} can be replaced by the (weaker) bound
\begin{equation}\label{alternateTrilinBound}
\Big|\bigcup Y(T)\Big|\geq c_\epsilon c_s \lambda^{3+1/4}K^{-1} \theta\delta^{3/4+\epsilon}(\delta^3|\tubes|).
\end{equation}
\end{remark}

\begin{proof}[Proof of Corollary \ref{volumeTrilinearCor}]
Fix a value of $\epsilon>0$. By choosing the constant $c_{\epsilon}$ in Corollary \ref{volumeTrilinearCor} sufficiently small, we can assume that $|\log\delta|\leq\delta^{-\epsilon/13}$. Let $c_{\epsilon}\leq C_{\epsilon/9}^{-9/4}$.

After a refinement of $(\tubes,Y)$ (obtained by replacing each shading $Y(T)$ by $Y(T)\cap B$), we can assume that $\sum\chi_{Y(T)}\sim \mu\chi_B$; that \eqref{notConcentratedInThetaPlane} still holds; and that $(\tubes,Y)$ is still $s$ robustly transverse with error $1/100$.

By \eqref{notConcentratedInThetaPlane} and the fact that the tubes are robustly transverse, for each point $x\in B$ we have
$$
\sum_{T_1,T_2,T_3\in\tubes} \chi_{T_1}(x)\ \chi_{T_2}(x) \chi_{T_3}(x) |v_1\wedge v_2\wedge v_3|^{12/13}\gtrsim \theta^{12/13}\mu^3,
$$
where the implicit constant depends on $s$. To see this, note that there are $\gtrsim\mu$ choices for $T_1$ with $x\in Y(T_1)$. Next, since $(\tubes,Y)$ is $s$--robustly transverse with error $1/100$, there are $\gtrsim \mu$ choices for $T_2$ with $x\in Y(T_2)$ and $\angle(T_1,T_2)\geq s$. Finally, by \eqref{notConcentratedInThetaPlane}, there are $\gtrsim \mu$ choices for $T_3$ with $x\in Y(T_3)$ such that the angle between $T_3$ and the plane spanned by $T_1$ and $T_2$ is $\geq \theta$. For each such choice of $T_1,T_2,T_3$, we have $v_1\wedge v_2\wedge v_3\gtrsim s\theta.$ 

Thus
\begin{equation*}
\int_{B}\Big( \sum_{T_1,T_2,T_3\in\tubes} \chi_{T_1}(x)\ \chi_{T_2}(x) \chi_{T_3}(x)\ (v_1\wedge v_2\wedge v_3)^{12/13}\Big)^{13/27}\gtrsim \theta^{4/9}|B|\mu^{13/9}.
\end{equation*}
On the other hand, by Theorem \ref{3linKakeyaR4} we have
\begin{equation*}
\int_{B}\Big( \sum_{T_1,T_2,T_3\in\tubes} \chi_{T_1}(x)\ \chi_{T_2}(x) \chi_{T_3}(x) |v_1\wedge v_2\wedge v_3|^{12/13} \Big)^{13/27}\leq C_{\epsilon/2} \delta^{-1/3-\epsilon/9}K_{\tubes, d(\epsilon/9)}^{1/9}(\delta^3|\tubes|)^{4/3}.
\end{equation*}

We conclude that 
$$
\theta^{4}|B|^9\mu^{13}\leq C_{\epsilon/2}^9 \delta^{-3-9\epsilon/9}K_{\tubes, d(\epsilon/9)}(\delta^3|\tubes|)^{12}.
$$

Since $|B|\geq |\log\delta|^{-1}\lambda\mu^{-1}(\delta^3|\tubes|)\geq \delta^{-\epsilon/13}\lambda\mu^{-1}(\delta^3|\tubes|)$, we have
\begin{equation}
\begin{split}
\Big|\bigcup Y(T)\Big|&\geq \Big(C_{\epsilon/9}^{-9}\lambda^{13}\delta^{3+9\epsilon/9}K_{\tubes, d(\epsilon/9)}^{-1}(\delta^3|\tubes|\delta^{-\epsilon})  \Big)^{1/4} \\
&\geq c_{\epsilon} c_s \lambda^{13/4}K^{-1/4} \theta\delta^{3/4}(\delta^3|\tubes|)^{1/4}.
\end{split}
\end{equation}
\end{proof}

The rest of Section \ref{trilinearKakeyaR4Sec} will be devoted to proving Theorem \ref{3linKakeyaR4}. 
\subsection{Main tools}
\subsubsection{Multilinear Kakeya}
We will make use of the multilinear Kakeya bounds discussed in the introduction. Specifically, we will use a slightly technical version that was established by Bourgain and the first author in \cite{BG}. 
\begin{theorem}[Bourgain-Guth, \cite{BG}, Theorem 6]\label{BourgainGuthTHhm}
Let $\tubes_1,\tubes_2,\tubes_3$ be three sets of $\delta$--tubes in $\RR^4$. Then
\begin{equation}
\Big(\int \sum_{T_1,T_2,T_3\in\tubes_1\times\tubes_2\times\tubes_3}\chi_{T_1}\ \chi_{T_2}\ \chi_{T_3}\ v_1\wedge v_2\wedge v_3 \Big)^{1/2}\lesssim \delta^4 (|\tubes_1|\ |\tubes_2|\ |\tubes_3|)^{1/2}.
\end{equation}
\end{theorem}

\begin{remark} Theorem 6 from \cite{BG} states this result for the case $|\tubes_1|=|\tubes_2|=|\tubes_3|$, but the proof gives the estimate above when the three cardinalities are not equal. \end{remark}
\subsubsection{Grains decomposition}

\begin{defn}[Grain]
A grain (of complexity $d$) in $\RR^n$ is the $C\delta$--neighborhood of a semi-algebraic set of dimension $\leq n-1$ and complexity $\leq d$.
\end{defn}
\begin{remark}
One can also define grains of dimension smaller than $n-1$. Such grains may play a role in establishing $k$--linear estimates in $\RR^n$ when $k<n-1$.
\end{remark}

\begin{defn}\label{cubeDefn}
A $\delta$--cube in $\RR^n$ is a set of the form $[0,\delta)^n+ v,$ where $v\in(\delta\ZZ)^n$
\end{defn}

\begin{defn}[Grains decomposition]
Let $\QQQ$ be a set of cubes in $\RR^n$. A \emph{Grains decomposition} of $\QQQ$ of degree $d$ and error $\epsilon$ is a set of grains $\mathcal{G}$, each of which has complexity $\leq d$. This decomposition has the following properties.
\begin{itemize}
\item The cubes are evenly-distributed across the grains: For each $G\in\mathcal{G}$, there is a set $\QQQ_G\subset \QQQ$. The sets $\{\QQQ_G\}$ are disjoint; if $Q\in\QQQ_G$ then $Q\subset G$. We have 
\begin{equation}\label{mostCubesCaptured}
\sum_{G\in\mathcal{Q}}|\QQQ_G|\geq\delta^\epsilon|\QQQ|, 
\end{equation}
and
\begin{equation}\label{grainsSameSize}
\delta^{\epsilon}|\QQQ|/|\mathcal{G}|\leq|\QQQ_G|\leq\delta^{-\epsilon}|\QQQ|/|\mathcal{G}|.
\end{equation}

\item A tube doesn't intersect too many grains: If $T$ is a $\delta$--tube, then
\begin{equation}\label{notTooManyTubesPerGrain}
|\{G\in\mathcal{G}\colon T\cap Q\neq\emptyset\ \textrm{for some}\ Q\in\QQQ_G\}|\leq C_\epsilon\delta^{-\epsilon} |G|^{1/n}.
\end{equation}
\end{itemize}
\end{defn}

\begin{prop}\label{existsGrainsDecomp}
Let $\QQQ$ be a set of $\delta$-cubes in $\RR^n$. Then for each $\epsilon>0$, there exists a grains decomposition of $\QQQ$ of degree $d(\epsilon)$ and error $\epsilon$.
\end{prop}

Before we prove Proposition \ref{existsGrainsDecomp}, first recall the polynomial partitioning theorem from \cite{GK}:
\begin{theorem}[Polynomial partitioning]
Let $U\subset\RR^n$ be an open set. Then for each $d\geq 1$, there is a polynomial $P\in\RR[x_1,\ldots,x_n]$ of degree $\leq d$ so that $\RR^n\backslash Z(P)$ is a union of $A\sim d^n$ cells (unions of open connected components of $\RR^n\backslash Z(P)$), and for each cell $O$, $|O\cap U|=|U|A^{-1}$.
\end{theorem}

\begin{cor}\label{equivsalg} Let $\QQQ$ be a set of $\delta$-cubes in $\RR^n$ and let $d>0$. Then one of the following two things must happen:

(Cellular case) There is a (non-zero) polynomial $P$ of degree at most $d$ so that the following holds.  The set $\RR^n \setminus Z(P)$ is a disjoint union of $\sim d^{n}$ open sets $O_i$.  For each $O_i$, define $\QQQ_i \subset \QQQ$ as the set of cubes $Q \in \QQQ$ contained in $O_i \setminus N_{10 \delta} (\partial O_i)$.  For each $i$, $| \QQQ_i | \lesssim d^{-n} | \QQQ|$, but at the same time $\sum_i | \QQQ_i | \gtrsim |\QQQ|$.

(Algebraic case) There is an algebraic variety $Z$ of dimension $\leq n-1$ and degree at most $d$ so that $| \QQQ_{Z} | \gtrsim | \QQQ|$, where $\QQQ_{Z}$ is the set of cubes in $\QQQ$ that lie in $N_{C \delta} (Z)$ for some constant $C$.  

\end{cor}

We are now ready to prove Proposition \ref{existsGrainsDecomp}
\begin{proof}[Proof of Proposition \ref{existsGrainsDecomp}]
Apply Corollary \ref{equivsalg} to the set of cubes $\QQQ$.  If we are in the algebraic case, we stop.  If we are in the cellular case, then we apply the partitioning lemma again inside of each cell.  First we establish a little notation.  We denote the cells by $O_j$.  We define $O_j' := O_j \setminus N_{10 n \delta} (\partial O_j)$.  We let $\QQQ_j \subset \QQQ$ be the set of cubes fully contained in $O_j'$.  In the cellular case, by definition, $\sum_j | \QQQ_j | \gtrsim | \QQQ |$.  We assign a weight to each cell $O_j$ proportional to the number of cubes in $\QQQ_j$.  

In the cellular case, we apply Corollary \ref{equivsalg} to each $\QQQ_j$.  If the Lemma comes out in the algebraic case for a fraction $\gtrapprox 1$ of the cells (by weight), then we stop.  Otherwise, for each $\QQQ_{j_1}$ which comes out in the cellular case, we let $O_{j_1, j_2} \subset O_{j_1}$ be the subcells of $O_{j_1}$, and we define $O_{j_1, j_2}' := O_{j_1, j_2} \setminus N_{10 n \delta} (\partial O_{j_1, j_2})$, and we define $\QQQ_{j_1, j_2} \subset \QQQ_{j_1}$ to be the set of cubes contained in $O_{j_1, j_2}'$.  In the cellular case, $\sum_{j_1, j_2} | \QQQ_{j_1, j_2} | \gtrapprox | \QQQ|$.

We continue in this way until the Lemma comes out in the algebraic case for a fraction $\gtrapprox 1$ of the cells (by weight).  After $s$ steps, we have cells $O_J$ for $J = (j_1, ..., j_s)$, along with $O_J'$ and $\QQQ_J$ defined as above.  The size of $| \QQQ_J|$ is controlled by Corollary \ref{equivsalg}, 

\begin{equation} \label{|Q_J|} | \QQQ_J | \le (C d^{-n})^s | \QQQ |. \end{equation}

If we are not in the algebraic case, then we must have $| \QQQ_J | \ge 1$ for some $J$, and so we see that this procedure will stop at some $s \lesssim \log | \QQQ | \lessapprox 1$.

By Inequality \ref{|Q_J|}, we know that $(C d^{n})^s \le | \QQQ |$.  We choose $d = d(\epsilon)$ big enough that the contribution of $d$ dominates the contribution of $C$, and so $C^s \lessapprox 1$.

For each $J = (j_1, ..., j_s)$ for which Corollary \ref{equivsalg} comes out in the algebraic case, we let $Z_J$ be the $\leq(n-1)$-dimensional algebraic variety guaranteed by the corollary.  Let $G$ be the $C\delta$ neighborhood of $Z_J$. We let $\QQQ_{G} \subset \QQQ_J$ be the subset of cubes of $\QQQ_J$ that lie in $N_{C \delta} (Z_J)$.  By Corollary \ref{equivsalg}, we have $| \QQQ_{J,alg} | \gtrsim | \QQQ_J|$, and since the procedure stopped, we see $\sum_J | \QQQ_{G} | \gtrapprox | \QQQ |$.  We also recall that the cubes of $\QQQ_J$ lie in $O_J'$. After dyadic pigeonholing $|\QQQ_J|$, we can prune the set of indices $J$ so that for each remaining $J$, $| \QQQ_J | \sim t$ and we still have $\sum_J | \QQQ_{G} | \gtrapprox | \QQQ |$.  Since each $| \QQQ_J |$ obeyed $| \QQQ_J | \le (C d^{-n})^s | \QQQ| \lessapprox d^{-ns} | \QQQ |$, we see that $t \lessapprox d^{-ns} | \QQQ |$.  

Let $\mathcal{G}$ be the set of grains associated to these indices $J$. We have 
$$
|\mathcal{G}|\gtrapprox|\QQQ|/t\gtrapprox d^{ns}.
$$

Now, let $T$ be a $\delta$--tube. Note that if $O_{j_1,\ldots,j_t}^\prime$ is a cell that intersects $T$, then after we apply Corollary \ref{equivsalg} to $\QQQ_{j_1,\ldots,j_t}$, at most $d+1$ cells $O_{j_1,\ldots,j_t,j_{t+1}}^\prime$ intersect $T$. Thus $T$ intersects $\leq (d+1)^s\lessapprox|\mathcal{G}|^{1/n}$ cells total.
\end{proof}

We will only use Proposition \ref{existsGrainsDecomp} in the special case $n=4$.

\subsection{Proof of Proposition \ref{techVersion3lin}}

Here is an overview of the proof.  To estimate an expression of the form 

$$\int \left( \sum_{T_1,T_2,T_3\in\tubes}\chi_{T_1}\ \chi_{T_2}\ \chi_{T_3}\ |v_1\wedge v_2\wedge v_3|^p \right)^q,$$

we divide the domain into grains and we use trilinear Kakeya to estimate the contribution of each grain.  The resulting estimate is very strong when the grains are small and becomes weaker as the grains get bigger.  At one extreme, if each grain has diameter $\delta$, then we get very good bounds -- bounds as strong as the full Kakeya conjecture.  At the other extreme, if there is only one grain with maximal diameter, then the bounds that we get are only as good as trilinear Kakeya.  In this bad scenario, all the tubes would lie in a single large grain.  However, the polynomial Wolff axioms limit how many tubes can lie in a single grain, ruling out this bad scenario.  Because not too many tubes can lie in a single grain, we get an improvement over trilinear Kakeya.  

\begin{proof}[Proof of Proposition \ref{techVersion3lin}]
Fix $\epsilon>0$. Let $d(\cdot)$ be the function from Proposition \ref{existsGrainsDecomp}. Let $K=\sup_{1\leq E\leq d(\epsilon)}{K_E}$.

For simplicity of notation, we will replace each tube $T\in\tubes$ with the union of all $\delta$--cubes that intersect $T$; thus we should think of each $T\in\tubes$ as a union of cubes. 

Now we do some dyadic pigeonholing.  First we dyadic pigeonhole $|v_1 \wedge v_2 \wedge v_3|$.  There is a dyadic $\theta$ so that

$$\int\Big(\sum_{T_1,T_2,T_3\in\tubes}\chi_{T_1}\ \chi_{T_2}\ \chi_{T_3}\ |v_1\wedge v_2\wedge v_3|^{12/13}\Big)^{13/27} \sim \int\Big(\sum_{\substack{T_1,T_2,T_3\in\tubes\\ v_1\wedge v_2\wedge v_3 \sim\theta}
} \chi_{T_1}\ \chi_{T_2}\ \chi_{T_3}\ |v_1\wedge v_2\wedge v_3|^{12/13}\Big)^{13/27}. $$

Next we divide $B(0,2)$ into $\delta$-cubes $Q$ and we dyadic pigeonhole their contribution to the last integral.  We can choose a set of cubes $\QQQ$ with $\cup_{Q \in \QQQ} Q = A$ so that

$$\int_{B(0,2)} \Big(\sum_{T_1,T_2,T_3\in\tubes}\chi_{T_1}\ \chi_{T_2}\ \chi_{T_3}\ |v_1\wedge v_2\wedge v_3|^{12/13}\Big)^{13/27} \sim \int_A \Big(\sum_{\substack{T_1,T_2,T_3\in\tubes\\ v_1\wedge v_2\wedge v_3 \sim\theta}
} \chi_{T_1}\ \chi_{T_2}\ \chi_{T_3}\ |v_1\wedge v_2\wedge v_3|^{12/13}\Big)^{13/27}, $$

\noindent and so that each $Q \in \QQQ$ makes a roughly equal contribution to the right-hand side.  

Let $\mathcal{G}$ be a grains decomposition of $\QQQ$ with error $\epsilon$ and degree $d=d(\epsilon)$.  Next we consider how much a tube intersects a grain.  The intersection of a tube with a grain could have a few different connected components, perhaps with different lengths.  Since the grain is a semi-algebraic set of complexity at most $d$, the number of these components is $O_d(1)$, and there is no harm in treating them separately.  If $W\subset\RR^4$, then we write $\operatorname{CC}(W,x)$ for the Euclidean connected component of $W$ containing $x$.  If $G$ is a grain and $T$ is a tube, then note that $\operatorname{diam}(\operatorname{CC}(T\cap G,x ))$ is the length of the component of $T \cap G$ thru the point $x$.  Next we dyadic pigeonhole these lengths:  we can find $\ell_1,\ell_2,\ell_3$ so that
\begin{equation}\label{pigeonholeEllI}
(\operatorname{LHS}\ \eqref{trilinearBoundFixedP})\lessapprox\sum_{G\in\mathcal{G}}\int_{G\cap A} \Big(\sum_{\substack{T_1,T_2,T_3\in\tubes\\v_1\wedge v_2\wedge v_3 \sim\theta\\ \ell_i\leq \operatorname{diam}(\operatorname{CC}(T_i\cap G,x ))\leq 2\ell_i } } \theta^{12/13}\ \chi_{T_1}\ \chi_{T_2}\ \chi_{T_3}\Big)^{13/27}.
\end{equation} 

Without loss of generality, we can assume that $\ell_1=\max(\ell_1,\ell_2,\ell_3)$. Cover each grain $G$ by balls of radius $C\ell_1$ so that the balls are $O(1)$ overlapping, and any subset of $G$ of diameter $\leq 2\ell_1$ is entirely contained within one of the balls. The intersection of $G$ with a ball of this type will be called a sub-grain $G^\prime$ with parent $G$. If $G^\prime$ is a sub-grain of $G$, let $\QQQ_{G^\prime}=\{Q\in\QQQ_G\colon Q\subset G^\prime\}$. Let $\mathcal{G}^\prime$ denote the set of sub-grains.

Note that if $T$ is a $\delta$--tube and $G\in\mathcal{G}$ is a grain satisfying $T\cap G\neq\emptyset$, then there are $O_d(1)$ sub-grains $G^\prime\in\mathcal{G}^\prime$ with $G^\prime\subset G$ that contain a point $x\in T\cap G^\prime$ with $\operatorname{diam}(\operatorname{CC}(T_i\cap G),x)\leq 2\ell_i$. This is because $T\cap G$ is a semi-algebraic set of bounded-complexity, so it is a union of $O_d(1)$ connected sets (see \cite{BCR} for further information and background on semi-algebraic sets). Each of these sets with diameter $\leq 2\ell_i$ intersects at most $O(1)$ sub-grains $G^\prime$. Furthermore, if there exists $x\in T\cap G$ with $\ell_i\leq \operatorname{diam}(\operatorname{CC}(T_i\cap G),x)\leq 2\ell_i$, then there exists at least one sub-grain $G^\prime$ with $c_d\ell_i\leq \operatorname{diam}(\operatorname{CC}(T_i\cap G^\prime),x)\leq 2\ell_i$, where $c_d>0$ depends only on $d$. For each sub-grain $G^\prime$, define
$$
\tubes_{i,G^\prime}=\{T\cap G^\prime\colon\ \textrm{there exists a component}\ W\subset T\cap G\ \textrm{with}\ W\subset G^\prime,\ c_d\ell_i\leq \operatorname{diam}(W)\leq 2\ell_i\}.
$$

Thus for each $i=1,2,3$,
\begin{equation}\label{smallGrainsComparableLarge}
\sum_{G\in\mathcal{G}}|\{ T\in\tubes\colon \textrm{there exists a component}\ W\subset T\cap G\ \textrm{with}\ \ell_i\leq \operatorname{diam}(W)\leq 2\ell_i\}|\sim \sum_{G^\prime\in\mathcal{G}^\prime}|\tubes_{i,G^\prime}|.
\end{equation}

After dyadic pigeonholing and refining the grains in $\mathcal{G}^\prime$, we can assume that \eqref{mostCubesCaptured}, \eqref{grainsSameSize}, and \eqref{notTooManyTubesPerGrain} hold for $\mathcal{G}^\prime$ (doing this may have made $|\mathcal{G}^\prime|$ smaller by a factor of $\lessapprox 1$). Furthermore, since each grain $G^\prime\in\mathcal{G}^\prime$ is contained in a ball of radius $\ell_1$, we have 
\begin{equation}\label{grainSize}
|G^\prime|\lesssim \ell_1^3\delta
\end{equation}
for all $G^\prime\in\mathcal{G}^\prime$, where the implicit constant depends only on $d(\epsilon),$ which in turn depends only on $\epsilon$ (see e.g.~\cite{Won}).

After further dyadic pigeonholing the grains $G\in\mathcal{G}^\prime$, we can find numbers $N_1,N_2,N_3$ so that $N_i\leq |\tubes_{i,G^{\prime}}|\leq 2N_i$ for each $G^{\prime}\in\mathcal{G}^{\prime}$.  Since $\mathcal{G}^\prime$ obeys \eqref{notTooManyTubesPerGrain},
\begin{equation*}
N_i|\mathcal{G}^{\prime}|\lessapprox C_\epsilon\delta^{-\epsilon} |\tubes|\ |\mathcal{G}^{\prime}|^{1/4},
\end{equation*}
i.e.
\begin{equation}\label{boundOnNi}
N_i \lessapprox  C_\epsilon\delta^{-\epsilon}|\tubes|\ |\mathcal{G}^{\prime}|^{-3/4},\qquad i=1,2,3.
\end{equation}

Let $\mathcal{Q}^{\prime}=\bigcup_{\mathcal{G}^{\prime}}\QQQ_G$. After dyadic pigeonholing, we can find numbers $\mu$ and $\mu_1,\mu_2,\mu_3$ so that if we refine $\QQQ^{\prime}$ and the associated sets $\QQQ_{G^\prime}$, then if $Q\in\cup\QQQ_{G^\prime}$ for some $G^\prime\in\mathcal{G}^{\prime}$ and if $x\in Q$, then
$$
\sum_{\substack{(T_1,T_2,T_3)\in\tubes_{1,G^\prime}\times \tubes_{2,G^\prime}\times \tubes_{3,G^\prime}\\v_1\wedge v_2\wedge v_3 \sim\theta} }\chi_{T_1}(x)\chi_{T_2}(x)\chi_{T_3}(x)\sim \mu^3,
$$
and $\sim\mu_i$ tubes from $\tubes_i$ pass through $x$ for each $i=1,2,3$. Note that
\begin{equation}\label{muPrimeVsMu123}
\mu \leq(\mu_1\mu_2\mu_3)^{1/3}. 
\end{equation}
\begin{remark}
The LHS of \eqref{muPrimeVsMu123} might be much smaller than the RHS if (for example) the main contribution to \eqref{trilinearBoundFixedP} from a typical point comes from triples with $v_1\wedge v_2\wedge v_3\sim 1$, but at the same time most triples through a typical point have $v_1\wedge v_2\wedge v_3$ much smaller than 1.
\end{remark}

After a further refinement of $\mathcal{G}^{\prime}$ to throw away those grains for which $|\QQQ_G|$ is small, we can assume that $|\QQQ_G|\approx |\QQQ^{\prime}|/|\mathcal{G}^{\prime}|$ for all $G\in\mathcal{G}^{\prime}$. 

Let $A^{\prime}=\bigcup_{\mathcal{Q}^{\prime}}Q$. Then $|A^{\prime}|\approx |A|$.  Because of all the dyadic pigeonholing, we have

\begin{equation} \label{simpLHS} (\operatorname{LHS}\ \eqref{trilinearBoundFixedP}) \approx |A^{\prime}| \mu^{13/9} \theta^{4/9}. \end{equation}

To prove \eqref{trilinearBoundFixedP}, it suffices to show that

\begin{equation} \label{trilineargoal} |A^\prime|^9 \mu^{13} \theta^4 \lesssim K \delta^{-3 - 9 \eps} (\delta^3 | \tubes| )^12. \end{equation}

Let $i_0$ be the index so that $\mu_{i_0}=\max(\mu_1,\mu_2,\mu_3)$. So in particular, \eqref{muPrimeVsMu123} implies
\begin{equation}\label{muLeqMuI0}
\mu \leq\mu_{i_0}.
\end{equation}

We note that 
\begin{equation}
\begin{split}
\mu_i|A^{\prime}| &\lesssim \sum_{G^\prime\in\mathcal{G}^{\prime}} \sum_{T\in\tubes_{i,G}}|T\cap A^{\prime}\cap G^\prime|\\
&\leq \sum_{G^\prime\in\mathcal{G}^{\prime}} \sum_{T\in\tubes_{i,G^\prime}}|T\cap G^\prime|\\
&\leq  |\mathcal{G}^{\prime}|N_i\ell_i\delta^3\\
&\lessapprox  C_\epsilon\delta^{-\epsilon}|\mathcal{G}^{\prime}|^{1/4}\ell_i(\delta^3|\tubes|),
\end{split}
\end{equation}
where on the last line we used \eqref{boundOnNi}. Thus
$$
\ell_i\gtrapprox  c_\epsilon\delta^{\epsilon}\mu_i\ |A^{\prime}|\ |\mathcal{G}^{\prime}|^{-1/4}(\delta^3|\tubes|)^{-1}.
$$
In particular, since $\ell_1=\max(\ell_1,\ell_2,\ell_3)$, we have
\begin{equation}\label{ell1MuI0}
\ell_1\gtrapprox  c_\epsilon\delta^{\epsilon}\mu_{i_0}|A^\prime|\ |\mathcal{G}^{\prime}|^{-1/4}(\delta^3|\tubes|)^{-1}.
\end{equation}

On the other hand, since the tubes from $\tubes$ satisfy the polynomial Wolff axioms, by \eqref{grainSize} we have
\begin{equation}\label{boundOnN1}
\begin{split}
N_1& \lesssim K(\ell_1^3)(\delta)\delta^{-3}\ell_1^{-4}\\
&=K\ell_1^{-1}\delta^{-2}\\
&\lessapprox C_\epsilon\delta^{-\epsilon}K\mu_{i_0}^{-1}|A^\prime|^{-1}\ |\mathcal{G}^{\prime}|^{1/4}(\delta^3|\tubes|)\delta^{-2}.
\end{split}
\end{equation}

\begin{remark}\label{whereWolffAxiomsAreUsed}
Observe that we have not used the full strength of the polynomial Wolff axioms. Instead, we have only used the fact that if $Z\subset\RR^4$ is a hypersurface of degree at most $d$, and if $B(x,r)$ is a ball of radius $r$, then 
\begin{equation}\label{restrictedPolyWolffAxioms}
|\{T\in\tubes\colon |T\cap B(x,r)\cap N_{\delta}(Z)|\gtrsim r|T|\}|\leq K_d r^{-1}\delta^{-2}.
\end{equation}
This condition can even be weakened slightly further: we can replace the bound \eqref{restrictedPolyWolffAxioms} with the requirement that for each $w>0$, there exists a constant $K_{d,w}$ so that   
\begin{equation}\label{restrictedPolyWolffAxiomsWeakened}
|\{T\in\tubes\colon |T\cap B(x,r)\cap N_{\delta}(Z)|\gtrsim r|T|\}|\leq K_{d,w} r^{-1}\delta^{-2-s}.
\end{equation}

The precise assumptions on $\tubes$ needed to prove Theorem \ref{mainThm} will be discussed further in Section \ref{minimalConditions} below.
\end{remark}

Let $G^\prime\in\mathcal{G}^\prime$. We have
\begin{equation}
\begin{split}
|A^\prime|\ |\mathcal{G}^\prime|^{-1}\mu^{3/2}\theta^{1/2}&\approx\int_{G^\prime\cap A^\prime} \Big(\sum_{\substack{T_1,T_2,T_3\in\tubes_{1,G^\prime}\times\tubes_{2,G^\prime}\times\tubes_{3,G^\prime}\\v_1\wedge v_2\wedge v_3 \sim\theta\\ } }\theta \chi_{T_1}\chi_{T_2}\chi_{T_3}\Big)^{1/2}\\
&\lesssim \delta^4(N_1N_2N_3)^{1/2}\\
&\lessapprox \delta(\delta^3|\tubes|) |\mathcal{G}^\prime|^{-3/4} N_1^{3/8}N_1^{1/8}\\
&\lessapprox \delta (\delta^3|\tubes|)|\mathcal{G}^\prime|^{-3/4} ( |\tubes|\ |\mathcal{G}^\prime|^{-3/4})^{3/8}(C_\epsilon\delta^{-\epsilon}\mu_{i_0}^{-1}K|A^\prime|^{-1}|\mathcal{G}^\prime|^{1/4}(\delta^3|\tubes|)\delta^{-2})^{1/8}\\
&=C_\epsilon\delta^{-\epsilon/8}\delta^{-3/8}|\mathcal{G}^\prime|^{-1}\mu_{i_0}^{-1/8}|A^\prime|^{-1/8}(\delta^3|\tubes|)^{3/2}K^{1/8},
\end{split}
\end{equation}
where on the third line we used Theorem \ref{BourgainGuthTHhm}, on the fourth line we used \eqref{boundOnNi}, and on the fifth line we used \eqref{boundOnNi} and \eqref{boundOnN1}.

Re-arranging and using \eqref{muLeqMuI0} and the fact that $|A^\prime|\approx|A|$, we have 
$$
|A|^9\mu^{13}\theta^4\lessapprox \delta^{-9\epsilon/8} K\delta^{-3}K(\delta^3|\tubes|)^{12}.
$$
Thus if we choose $C^\prime_{\epsilon}$ sufficiently large (depending only on $\epsilon$), then
\begin{equation}
|A|^9\mu^{13}\theta^4\leq C^\prime_{\epsilon} K\delta^{-3-9\epsilon}K(\delta^3|\tubes|)^{12}.
\end{equation}
This establishes \eqref{trilineargoal} and completes the proof of Proposition \ref{techVersion3lin}.
\end{proof}
\section{Volume bounds for unions of plainy tubes}\label{plainyTubesSec}
In this section we will prove the following volume bound for unions of tubes that satisfy the linear Wolff axioms.

\begin{theorem}\label{plainyTubesBound}
Let $\tubes$ be a set of $\delta$--tubes in $\RR^4$. For each $T\in\tubes$, let $Y(T)\subset T$ with $|Y(T)|\geq \lambda|T|$. Suppose that each tube satisfies the two-ends condition with exponent $\epsilon_0$ and error $\alpha$, and that $\tubes$ satisfies the following two properties: 
\begin{itemize}
\item Linear Wolff axioms: For each rectangular prism $S$ of dimensions $1\times t_1\times t_2\times t_3$ (with arbitrary orientation), we have
\begin{equation}\label{nonConcentrationOnPrisms}
|\{T\in\tubes\colon T\subset S\}|\leq K t_1t_2t_3\delta^{-3}.
\end{equation}
 
\item $\theta$--planiness:
There is a number $\delta\leq\theta $ so that for each $x\in\RR^4$ there is a two-dimensional plane $\Pi$ containing $x$ so that
\begin{equation}\label{quantitativeTrilinearity} 
\angle(T,\Pi)\leq\theta\ \textrm{for all}\ T\in\tubes\ \textrm{with}\ x\in Y(T).
\end{equation}

\end{itemize}

Then 
\begin{equation} \label{volumeBound1}
\Big|\bigcup_{T\in\tubes}Y(T)\Big| \ge c_{\epsilon,\epsilon_0}\alpha^{-C^\prime_{\epsilon,\epsilon_0}}\lambda^{3}K^{-1}\theta^{-1/9}\delta^{1+\epsilon}(\delta^3|\tubes|).\end{equation}
\end{theorem}
\begin{remark}
Taking $\theta=\delta$, Theorem \ref{plainyTubesBound} heuristically says that if $\tubes$ is a set of $\delta^{-3}$ tubes that satisfy the linear Wolff axioms, and if the tubes passing through a typical point lie in the $\delta$--neighborhood of a plane, then the union of the tubes has volume $\geq\delta^{8/9}$. This corresponds to a Hausdorff dimension bound of $3 + 1/9$. 
\end{remark}

\begin{remark}
We will actually prove a bound with a slightly better dependence on $\lambda$ than the bound given by \eqref{volumeBound1}. This phenomena often arises when one considers a collection of tubes that satisfy the two-ends condition. 
\end{remark}

\begin{defn}
We make the following notation for the remainder of Section \ref{plainyTubesSec}.  If we have a bound of the form 

\begin{equation}\label{newDefnGtrapprox}
A \ge c_{\epsilon,\epsilon_0}\alpha^{-C^\prime_{\epsilon,\epsilon_0}} \delta^\eps B,
\end{equation}

\noindent which holds for any $\epsilon, \epsilon_0 > 0$ and $\alpha$, then we write

$$ A \gtrapprox_{\eps_0,\alpha} B.$$

In the above expression, $\eps>0$ is arbitrary, while $\eps_0$ and $\alpha$ quantify the extent to which the collection of tubes (and their associated shading) satisfies the two-ends condition. Whenever the symbol $\gtrapprox$ is used, the set of tubes under consideration will be clear from context.  

For example, \eqref{volumeBound1} can be abbreviated as

$$ \Big|\bigcup_{T\in\tubes}Y(T)\Big| \gtrapprox_{\eps_0,\alpha} \lambda^{3}K^{-1}\theta^{-1/9}\delta (\delta^3|\tubes|). $$

\end{defn}

\subsection{Reduction to the $K=1$ case}
\begin{prop}\label{badKPlainyBound}
Let $\tubes$ be a set of $\delta$--tubes in $\RR^4$ that satisfy the hypotheses of Theorem \ref{plainyTubesBound} with $K=1$ in \eqref{nonConcentrationOnPrisms}. Then
\begin{equation} \label{volumeBoundBadK}
\Big|\bigcup_{T\in\tubes}Y(T)\Big|\gtrapprox_{\eps_0,\alpha} \lambda^{3}\theta^{-1/9}\delta(\delta^3|\tubes|).
\end{equation}
\end{prop}
\begin{proof}[Proof of Theorem \ref{plainyTubesBound}, using Proposition \ref{badKPlainyBound}]
First, observe that there exists a set $\mathcal{R}$ of rectangular prisms in $\RR^4$ so that $|\mathcal{R}|\leq \delta^{-64}$ and for any prism $R\subset B(0,2)$ of dimensions $1\times t_1\times t_2\times t_3$ with $\delta\leq t_i\leq 2,\ i=1,2,3$, there is a prism $R^\prime\in \mathcal{R}$ with $\frac{1}{2}R^\prime\subset R\subset 2R^\prime$. For example, let $\mathcal{R}$ be the set of prisms whose 16 corners lie in $(\delta\ZZ)^{4}$. 

Next, let $\tubes^\prime\subset\tubes$ be obtained by randomly selecting each tube in $\tubes$ with probability $K^{-1}/C$, where $C$ is a large constant. Note that each prism in $\mathcal{R}$ contains $\leq Kt_1t_2t_3\delta^{-3}$ tubes from $\tubes.$ Thus, if $C$ is chosen sufficiently large, then with high probability, each of the prisms in $\mathcal{R}$ contains $\leq t_1t_2t_3\delta^{-3}$ tubes from $\tubes^\prime$.

Apply Proposition \ref{badKPlainyBound} to $\tubes^\prime$. We conclude that
$$
\Big|\bigcup_{T\in\tubes}Y(T)\Big|\geq \Big|\bigcup_{T\in\tubes^\prime}Y(T)\Big| \gtrapprox_{\eps_0,\alpha} \lambda^{3}\theta^{-1/9}\delta(\delta^3|\tubes|/(CK)).\qedhere
$$
\end{proof}

Before proving Theorem \ref{plainyTubesBound}, we will establish a key estimate on the volume of unions of tubes, in the special case where the ``hairbrush'' of a typical tube can be contained in the union of a small number of planes.

\subsection{Small hairbrush volume estimates} 
Wolff's ``hairbrush'' argument shows that the union of a set of $\theta^{-3}$ $\theta$--tubes in $\RR^4$ satisfying the linear Wolff axioms must have volume  $\gtrapprox \theta$. In brief, the argument is as follows: If the union of the tubes has small volume, then many tubes must pass through each point of the union. Let $T_0$ be a typical tube, and consider the set of all tubes intersecting $T_0$; the set of tubes intersecting $T_0$ is called the hairbrush of $T_0$. If there are many tubes passing through each point of the union (and thus many tubes passing through each point of $T_0$), then the hairbrush must have large cardinality. However, the tubes in the hairbrush of $T_0$ are almost disjoint. This implies that their union must be large, and thus the union of all of the tubes must be large. 

In the following proposition, we will show that if the hairbrush of each tube has a certain special property, then we can use Wolff's hairbrush argument to get a stronger conclusion. The special property is that the hairbrush of each tube is contained in a union of planes, and the combined volume of this union is small. A precise version is stated below.

\begin{defn}
Let $(\tubes,Y)$ be a set of tubes. If $T_0\in\tubes$, we define
$$
\operatorname{Hair}(T_0)=\{T\in\tubes\colon Y(T_0)\cap Y(T)\neq\emptyset\}.
$$
\end{defn}

\begin{prop}\label{smallHairbrushProp}
Let $(\tubes,Y)$ be a set of essentially distinct $\theta$ tubes in $\RR^4$, and suppose $|Y(T)|\geq\lambda|T|$ for each tube. Suppose that each tube satisfies the two-ends condition with exponent $\epsilon_0$ and error $\alpha$, and
\begin{itemize}
 \item At most $Et_1t_2t_3 \theta^{-3}$ tubes from $\tubes$ are contained in a rectangular prism of dimensions $1\times t_1\times t_2\times t_3$.
 \item For each $T_0\in\tubes$ and each plane $\Pi$ with $T_0\subset N_{\theta}(\Pi)$, there are at most $D$ tubes in $\{T\in\operatorname{Hair}(T_0)\colon T\subset N_{\theta}(\Pi)\}$ pointing in the same direction (we say that two tubes point in the same direction if the lines coaxial with the tubes make an angle $\leq\theta$).
 \item For each $T\in\tubes$, with central line $L$, the hairbrush $\operatorname{Hair}(T)$ is contained in a union of planes containing $L$, whose intersection with $B(0,2)$ has combined volume $\leq \rho$.
\end{itemize}

Then after a refinement of $(\tubes,Y)$, we have the pointwise bound
\begin{equation}\label{smallHairbrushCoarseEstimate}
\sum_{T_\theta\in\tubes_\theta}\chi_{Y(T_\theta)} \lessapprox_{\eps_0,\alpha}  \lambda^{-1}\theta^{-1} E^{1/2}D^{1/4}\rho^{1/4}(\theta^3|\tubes|)^{1/4}.
\end{equation}
\end{prop}

We first need to reduce to the case where most pairs of intersecting tubes point in two linearly independent directions:
\begin{prop}\label{smallHairbrushPropRobustTrans}
For every $\epsilon>0$, there is a constant $C_{\epsilon}$ so that the following holds. Let $(\tubes,Y)$ be a set of essentially distinct $\theta$ tubes in $\RR^4$, and suppose $|Y(T)|\geq\lambda|T|$ for each tube. Suppose that $(\tubes,Y)$ is $s$-robustly transverse (where $s$ is a small constant that depends only on $\epsilon$ and $\epsilon_0$) with error $1/100;$ each tube in $\tubes$ satisfies the two-ends condition with exponent $\epsilon_0$ and error $\alpha$; and 
\begin{itemize}
 \item At most $Et_1t_2t_3 \theta^{-3}$ tubes from $\tubes$ are contained in a rectangular prism of dimensions $1\times t_1\times t_2\times t_3$.
 \item For each $T_0\in\tubes$, each plane $\Pi$ with $T_0\subset N_{\theta}(\Pi)$, there are $\leq D$ tubes $\{T\in\operatorname{Hair}(T_0)\colon T\subset N_{\theta}(\Pi)\}$ pointing in the same direction.
 \item For each $T\in\tubes$, with central line $L$, the hairbrush $\operatorname{Hair}(T)$ is contained in a union of planes containing $L$, whose intersection with $B(0,2)$ has combined volume $\leq \rho$. \end{itemize}

Then after a refinement of $(\tubes,Y)$, we have the pointwise bound
\begin{equation}\label{smallHairbrushCoarseEstimateRobustTrans}
\sum_{T\in\tubes}\chi_{Y(T)} \lessapprox_{\eps_0,\alpha} \lambda^{-1}\theta^{-1} E^{1/2}D^{1/4}\rho^{1/4}(\theta^3|\tubes|)^{1/4}.
\end{equation}
\end{prop}
\begin{proof}[Proof sketch of Proposition \ref{smallHairbrushProp} using Proposition \ref{smallHairbrushPropRobustTrans}]
The proof of Proposition \ref{smallHairbrushProp} using Proposition \ref{smallHairbrushPropRobustTrans} closely mirrors the proof of Proposition \ref{messyMainProp} using Proposition \ref{mainPropQuantTrans}, so we will only provide a brief sketch. As in the proof of Proposition \ref{messyMainProp}, we obtain a dichotomy: either most of the tubes passing through most points of $\bigcup_{T\in\tubes}Y(T)$ are $s$--robustly transverse, or the tubes in $\tubes$ can be partitioned into sets, each of which is contained in an $s$--cylinder, and whose unions are almost disjoint. In the former case, we apply Proposition \ref{smallHairbrushPropRobustTrans}, which immediately proves \ref{smallHairbrushProp}. In the latter case, we re-scale each of the cylinders, and apply Proposition \ref{smallHairbrushProp} to the corresponding set of $\delta/s$--tubes.

The main thing to observe is how the hypotheses of Proposition \ref{smallHairbrushProp} change under the rescaling from $\delta$ to $\delta/s$. In short, the parameters $E$, $D$, $\lambda$ and $\epsilon_0$ remain unchanged, while the parameter $\alpha$ is replaced by $C\alpha$ (where $C\lesssim 1$ is a large constant), $\theta$ is replaced by $\theta/s$, and $\rho$ is replaced by $\rho/s^2$. The key observation is that 
$$
(\rho/s^2)^{1/4}(\theta/s)^{-1}=s^{1/2}\rho^{1/4}\theta^{-1}.
$$ 
Since exponent of $s$ is positive, this re-scaling process gives us a stronger pointwise bound in \eqref{smallHairbrushCoarseEstimate}, which is what allows the ``induction step'' of the re-scaling argument to proceed.
\end{proof}

\begin{proof}[Proof of Proposition \ref{smallHairbrushPropRobustTrans}]
For each $\eps>0$, we need to establish the bound
\begin{equation}\label{smallHairbrushCoarseEstimateQuant}
\sum_{T_\theta\in\tubes_\theta}\chi_{Y(T_\theta)} \leq \big(C_{\epsilon,\epsilon_0}\alpha^{C^\prime_{\epsilon,\epsilon_0}} \theta^{-\eps}\big) \lambda^{-1}\theta^{-1} E^{1/2}D^{1/4}\rho^{1/4}(\theta^3|\tubes|)^{1/4}
\end{equation}
for some constants $C_{\epsilon,\epsilon_0}$ and $C^\prime_{\epsilon,\epsilon_0}$. 

Since the tubes in $\tubes_\theta$ are essentially distinct, we have $\sum_{T_\theta\in\tubes_\theta}\chi_{Y(T_\theta)}\leq |\tubes_\theta|\leq\theta^{-6}$. Thus by requiring that $C^\prime_{\epsilon,\epsilon_0}\geq24/\epsilon\epsilon_0$, we can assume that 
\begin{equation}\label{alphaSmall}
\alpha\leq \theta^{-\epsilon \epsilon_0/4}, 
\end{equation}
since otherwise Proposition \ref{smallHairbrushPropRobustTrans} follows immediately. 

Similarly, by requiring that $C_{\epsilon,\epsilon_0}$ be sufficiently large (independent of $\theta$), we can assume that 
\begin{equation}\label{logSmall}
|\log\theta|\leq\theta^{-\epsilon \epsilon_0/4}.
\end{equation}  

After a refinement of the shadings $Y(T)$, we can assume that $\sum_{T\in\tubes}\chi_{Y'(T)}\sim\mu_\theta\chi_B$, with $\mu_{\theta}|B|\geq |\log\theta|^{-1}\lambda(\theta^3|\tubes|)$. After this refinement, the tubes are still $s$--robustly transverse with error $1/100$, and they satisfy the two-ends condition with exponent $\epsilon_0$ and error $\alpha|\log\theta|$. It remains to show that 
$$
\mu_\theta \lessapprox_{\eps_0,\alpha} \lambda^{-1}\theta^{-1} E^{1/2}D^{1/4}\rho^{1/4}(\theta^3|\tubes|)^{1/4}.
$$

We begin as in Wolff's hairbrush argument \cite{W} by finding a tube with a large hairbrush.  The proof shows moreover that the hairbrush contains a large set that is not too close to the central tube.

\begin{lemma}\label{bigHairbrush}
There exists a tube $T_0\in \tubes_{\theta}$ and radius $r \geq\theta^{\eps/2}\gtrapprox_{\eps_0,\alpha} 1$ (here $\eps$ is the constant from \eqref{smallHairbrushCoarseEstimateQuant}) so that
\begin{equation}\label{hairBrushBd}
\left| \bigcup_{T\in\operatorname{Hair}(T_0)} Y'(T) \setminus N_r(T_0) \right| \gtrapprox_{\eps_0,\alpha} \lambda^2\theta^{2}\mu_\theta \max(E^{-1},D^{-1}). 
\end{equation} 
\end{lemma}
\begin{proof}
The proof is essentially Wolff's hairbrush argument from  \cite{W}. Let $\lambda^\prime = \frac{|\log\theta|^{-1}}{100}\lambda$ and let
$$
\tubes^{\prime}=\{T\in\tubes\colon |Y'(T)|\geq \lambda^\prime |T|\}.
$$
Then $|\tubes^\prime|\geq\frac{|\log\theta|}{100}|\tubes|$. Since the tubes in $(\tubes_\theta,Y)$ are $s$--robustly transverse with error $1/100$,
$$
|\{(T, T^\prime)\colon Y'(T)\cap Y'(T^\prime)\neq\emptyset,\ \angle(T,T^\prime)\geq s,\ T^\prime\in\tubes^\prime\}|\gtrapprox |\tubes| \lambda \theta^{-1}\mu_\theta.
$$
Thus by pigeonholing, there exists $T_0\in\tubes$ with
$$
|\{T^\prime\in\tubes^\prime\colon Y'(T_0)\cap Y'(T^\prime)\neq\emptyset,\ \angle(T,T^\prime)>s\}|\gtrapprox\lambda \theta^{-1}\mu_{\theta}.
$$
Since each tube from $\tubes$ (with the shading $Y'$) satisfies the two-ends condition with exponent $\epsilon_0$ and error $\alpha|\log\theta|$, and since $|Y'(T)|\geq \frac{|\log\theta|^{-1}}{100}\lambda\geq \frac{|\log\theta|^{-1}}{100}|Y(T)|$, if we define
$$
r=(200\alpha|\log\theta|^2)^{-1/\epsilon_0}\ge \theta^{\epsilon/2},
$$ 
(here we used \eqref{alphaSmall} and \eqref{logSmall}) then for every ball $B(r)$ of radius $r$ we have

$$
|B(r)\cap Y^\prime(T)|\leq \alpha|\log\theta| r^{\epsilon_0}|Y(T)|\leq\frac{1}{2}|Y^\prime(T)|.
$$

Cover $\RR^4\backslash N_{r}(T_0)$ by the $\theta$--neighborhoods of planes containing the line concentric with $T_0$; these sets are $\lesssim \theta^{-2(\epsilon/2})=\theta^{-\epsilon}$ overlapping.  For each such slab, we study the collection of $Y'(T) \setminus N_r(T_0)$ lying in the slab, with $T \in \operatorname{Hair}(T_0)$.  In this collection, there are at most $D$ tubes pointing in each direction.  Therefore the typical multiplicity of these tubes is $\lessapprox D$, by the standard $L^2$ argument.  Also, each rectangular prism of dimensions $1\times t\times \theta\times\theta$ can contain at most $Et\theta^{-1}$ tubes.   Using this information, the $L^2$ argument also shows that the typical multiplicty of these tubes is $\lessapprox E$.
$$
\left| \bigcup_{T\in\operatorname{Hair}(T_0)}Y^\prime(T) \setminus N_r(T_0) \right| \gtrapprox_{\eps_0,\alpha} \lambda_1^2 \theta^3 (\mu_{\theta}\theta^{-1})\max(E^{-1},D^{-1})=\lambda^2\theta^{2}\mu_\theta \max(E^{-1},D^{-1}).\qedhere
$$
\end{proof}
Now we bring into play our special assumption that $\operatorname{Hair}(T_0)$ lies in a union of 2-planes with small volume.  Let $T_0$ be the tube given by the last lemma and let $L_0$ be the line coaxial with $T_0$. Let $\Pi_0$ be a plane that intersects $B(0,2)$ with $\dist\big(T_0\cap B(0,2),\ \Pi_0\cap B(0,2)\big)\sim 1$. Let $\psi\colon\RR^4 \setminus L_0 \to \Pi_0$ be the projection.  Given a point $x \in \RR^4 \setminus L_0$, $\psi(x)$ is the unique intersection point between $\Pi_0$ and the plane spanned by $x$ and $L_0$.  For each point $y \in \Pi_0$, the preimage $\psi^{-1}(y)$ is the 2-plane spanned by $y$ and $L_0$.  (And for each line $l \subset \Pi_0$, the preimage $\psi^{-1}(l)$ is the 3-plane spanned by $l$ and $L_0$.)  The map $\psi$ is not Lipschitz -- as $x$ approaches $L_0$, $| \nabla \psi(x)|$ goes to infinity.  However, if we restrict $\psi$ to a map from $B(0,2) \setminus N_r(T_0)$ to $\Pi_0$, then $\psi$ has Lipschitz constant $\approx_{\eps_0,\alpha} 1$.  By assumption, $\bigcup_{T \in \operatorname{Hair}(T_0)} Y'(T)$ lies in a union of 2-planes thru $L_0$ with total volume at most $\rho$.  Therefore,

$$  \Big| \psi \Big( \bigcup_{T \in \operatorname{Hair}(T_0)} Y'(T) \setminus N_r(T_0) \Big) \Big| \lessapprox_{\eps_0,\alpha} \rho. $$

We want to focus on the region $B(0,2) \setminus N_r(T_0)$, where the geometry of $\psi$ is well-behaved, so we again refine our shading a little, replacing each $Y'(T)$ by $Y'(T) \setminus N_r(T_0)$.  By the Wolff axioms, not many tubes can have a large intersection with $N_r(T_0)$, so this has a negligible effect.  

We denote 

$$
Z_{T_0}= \bigcup_{T\in\operatorname{Hair}(T_0)}Y'(T).
$$

We have just seen that

$$ | \psi(Z_{T_0}) | \lessapprox_{\eps_0,\alpha} \rho. $$ 

But by Lemma \ref{bigHairbrush}, we know that 

\begin{equation} \label{Zlower} |Z_{T_0}| \gtrapprox_{\eps_0,\alpha} \lambda^2\theta^{2}\mu_\theta D^{-1/2} E^{-1/2}.\end{equation}

Now we have

$$ \mu_{\theta} |Z_{T_0}| \sim \int_{Z_{T_0}} \sum_{T \in \tubes} \chi_{Y'(T)}. $$

If $Y'(T) \cap Z_{T_0}$ has volume $\theta^3 l$, then $\psi (Y'(T) \cap Z_{T_0}$ has area $\sim \theta l$.  Therefore,

$$ \int_{Z_{T_0}} \sum_{T \in \tubes} \chi_{Y'(T)} \lesssim \theta^2 \int_{\psi(Z_{T_0})} \sum_{T \in \tubes} \chi_{\psi(Y'(T))}. $$

By Cauchy-Schwarz,

$$ \theta^2 \int_{\psi(Z_{T_0})} \sum_{T \in \tubes} \chi_{\psi(Y'(T))} \lesssim \theta^2 \rho^{1/2} \Big\| \sum_{T \in \tubes} \chi_{\psi(Y'(T))} \Big\|_{L^2(\Pi_0)}. $$

We pause to estimate this $L^2$ norm.  Using Cordoba's two-dimensional Kakeya argument from \cite{Cor}, we have 

$$ \Big\| \sum_{T \in \tubes} \chi_{\psi(Y'(T))} \Big\|_2^2 = \sum_{T_1, T_2 \in \tubes} | \psi(Y'(T_1)) \cap \psi(Y'(T_2)) | = $$

$$ = \sum_{T_1 \in \tubes}\ \sum_{t \textrm{ dyadic}}\ \sum_{\substack{ \psi(Y'(T_2)) \textrm{ intersects }\\ \psi(Y'(T_1)) \textrm{ in angle } \sim t }} t^{-1} \theta^2. $$

If $\psi(Y'(T_2))$ intersects $\psi(Y'(T_1))$ in an angle $\sim t$, then $T_2$ is forced to lie in a $t \times 1 \times 1 \times 1$ slab.  This slab is the $t$-neighborhood of the 3-plane spanned by $L_0$ and the central line of $L_1$.  For each $T_1$, the number of such $T_2$ is bounded by $E t \theta^{-3}$.  Plugging in this bound, we see that 

$$ \Big\| \sum_{T \in \tubes} \chi_{\psi(Y'(T))} \Big\|_2 \lessapprox \theta^{-1/2} E^{1/2} | \tubes|^{1/2}. $$

Plugging this last estimate into our reasoning above, we see that 

$$ \mu_{\theta} |Z_{T_0}| \lessapprox_{\eps_0,\alpha} \theta^{3/2} \rho^{1/2} E^{1/2} | \tubes |^{1/2}. $$

Plugging in the lower bound for $|Z_{T_0}|$ in (\ref{Zlower}) and rearranging gives 

$$ \mu_\theta^2 \lessapprox_{\eps_0,\alpha} \lambda^{-2} \theta^{-1/2} \rho^{1/2} D^{1/2} E | \tubes |^{1/2}, $$

which is equivalent to the desired bound

$$ \mu_\theta \lessapprox_{\eps_0,\alpha} \lambda^{-1} \theta^{-1} \rho^{1/4} D^{1/4} E^{1/2} ( \theta^3 |\tubes|)^{1/4}. $$

\end{proof}

\subsection{Analyzing the set at two scales}\label{scaleThetaAnalysis}
We return to the proof of Proposition \ref{badKPlainyBound}. Let $(\tubes,Y)$ be the set of tubes from the statement of the proposition.  After a refinement, we can assume that $\sum_{T \in \tubes} \chi_{Y(T)} \sim \mu \chi_W$.  We want to prove a lower bound on $|W|$.  

We will begin by replacing each tube $T\in\tubes$ with its $\theta$--neighborhood. Define $\tubes_\theta=\{N_{\theta}(T)\colon T\in\tubes\}$, and for each $T_{\theta}\in \tubes_\theta,$ let $\tubes(T_\theta)=\{T\in\tubes\colon T\subset T_{\theta}\}$. Then there exists a subset $\tubes_{\theta}\subset \hat\tubes$ so that the sets $\{\tubes(T_\theta)\colon T_\theta\in\tubes_\theta\}$ are disjoint; for any two tubes $T_\theta,\ T^\prime_\theta\in\tubes_{\theta}$, the 20--fold dilate of $T_\theta$ does not contain $T^\prime_\theta$; and  
$$
\sum_{T_\theta\in\tubes_{\theta}} |\tubes(T_{\theta})| \geq  c |\tubes|,
$$
where $c>0$ is an absolute constant.
\begin{remark}\label{reasonFor20}
The reason we insist that the 20-fold dilates of the tubes in $\tubes_\theta$ be distinct is that later we will replace each tube in $\tubes_\theta$ with its 10-fold dilate, and we want these dilated tubes to still be essentially distinct.
\end{remark}
Let $\tubes^\prime = \bigcup_{T_\theta\in\tubes_\theta}\tubes(T_\theta).$ Observe that $\tubes^\prime$ still satisfies \eqref{nonConcentrationOnPrisms} and \eqref{quantitativeTrilinearity}. We will call elements of $\tubes_{\theta}$ ``fat tubes,'' and elements of $\tubes^\prime$ ``thin tubes.''

After pigeonholing (which induces a refinement of $\tubes_{\theta}$ and $\tubes^\prime$), we can assume that there is a number $A$ with $1\leq A\leq \theta^{-3}$ so that there are $A\theta^{-3}$ fat tubes, and each fat tube contains $\approx|\tubes|/(A\theta^{-3})$ thin tubes from $\tubes^\prime$. 
\begin{lemma}
$\tubes_{\theta}$ also obeys a version of the linear Wolff axioms: for each rectangular prism $R$ of dimensions $1\times t_1\times t_2\times t_3$, with $t_1,t_2,t_3\geq\theta$, the number of tubes from $\tubes_{\theta}$ contained in $R$ is $\lessapprox A t_1t_2t_3\theta^{-3}(\delta^3|\tubes|)^{-1}$, and this property continues to hold if we refine the set of tubes.
\end{lemma}
\begin{proof}
Suppose that $R$ contains $L$ fat tubes.  Each fat tube contains $\sim |\tubes| /(A \theta^{-3})$ thin tubes, and so $R$ contains $\gtrsim L |\tubes|/ ( A \theta^{-3})$ thin tubes. Since $\tubes$ satisfies the linear Wolff axioms, by \eqref{nonConcentrationOnPrisms} we have
\begin{equation*}
L|\tubes|/(A\theta^{-3})\lessapprox t_1t_2t_3\delta^{-3},
\end{equation*}
and thus
\begin{equation*}
 L\lessapprox A t_1t_2t_3\theta^{-3}(\delta^3|\tubes|)^{-1}.\qedhere
\end{equation*}
\end{proof}
\subsection{Fine scale estimates}
We will now examine a single fat tube. A key tool will be Wolff's Kakeya bound from  \cite{W}, which was mentioned in the introduction. 


\begin{theorem}[\cite{W}]\label{WolffThm}
Let $\tubes$ be a set of $\rho$--tubes in $\RR^4$ that satisfy the linear Wolff axioms (i.e. estimate \eqref{nonConcentrationOnPrisms} with $K=1$), and for each $T\in\tubes$, let $Y(T)\subset T$. Suppose $|Y(T)|\geq \lambda|T|$ for each $T\in\tubes$ and that each tube satisfies the two-ends condition with exponent $\epsilon_0$ and error $\alpha$. Then after a refinement of $(\tubes,Y)$, we have the pointwise bound
\begin{equation}\label{WolffBoundOnTubes}
\sum_{T\in\tubes}\chi_{Y(T)} \lessapprox_{\eps_0,\alpha} \lambda^{-1/2} \rho^{-1} ( \rho^{3}|\tubes|)^{1/3}.
\end{equation}
\end{theorem}
For each fat tube $T_{\theta}\in\tubes_{\theta}$, apply Theorem \ref{WolffThm} and the same re-scaling argument from Section \ref{rescalingArgumentsSec} to the tubes in $\tubes(T_\theta)$.  After refining $Y(T)$, we get

$$ \sum_{T \in \tubes(T_\theta)} \chi_{Y(T)} \sim \mu_{\operatorname{fine}} \chi_{B_{T_\theta}}. $$

\noindent where the multiplicity $\mu_{\operatorname{fine}}$ obeys the bound 

\begin{equation*}
\mu_{\ofine} \lessapprox_{\eps_0,\alpha} \lambda^{-1/2} (\delta/\theta)^{-1} \big( (\delta/\theta)^{3}|\tubes_\theta|\big)^{1/3}. \end{equation*}

Plugging in that $|\tubes_\theta| \sim |\tubes|/(A \theta^{-3})$ and simplifying, we get

\begin{equation}\label{pointwiseMuBdInsideTube}
\mu_{\ofine} \lessapprox_{\eps_0,\alpha} \lambda^{-1/2} \theta^{1 + \eps} \delta^{-1} A^{-1/3} ( \delta^3 |\tubes| )^{1/3}. 
\end{equation}

\begin{remark}
Instead of using Theorem \ref{WolffThm}, it would be temping to instead apply Theorem \ref{mainThm} at scale $\rho$ to obtain a seemingly stronger variant of \eqref{WolffBoundOnTubes}. The problem with this approach is that Theorem \ref{mainThm} has worse dependence on the size of $|\tubes|$. While the RHS of \eqref{WolffBoundOnTubes} contains the term $( \rho^{3}|\tubes|)^{1/3}$, if we used Theorem \ref{mainThm} then the corresponding term would be $\rho^{3}|\tubes|,$ and this would lead to inferior bounds.

\end{remark}

\subsection{Coarse scale estimates}
\subsubsection{Defining a shading on the fat tubes}
We will now define a shading $Y(T_\theta)$ on the fat tubes in $\tubes_\theta$. Let $\mathcal{Q}=\big\{[0,\theta)^4 + \theta v\colon v\in\ZZ^4 \big\}$; this is a set of disjoint $\theta$--cubes whose union is $\RR^4$. For each $T_\theta\in\tubes_\theta$, we can find a (finite) set $\mathcal{Q}_{T_\theta}\subset\mathcal{Q}$ and a number $w_{T_\theta}$ so that for each $Q\in\mathcal{Q}_{T_\theta},$
$$
\sum_{T\in\tubes(T_\theta)}|Q\cap Y^\prime(T)|\sim w_{T_\theta},
$$
and 
$$
\sum_{Q\in\mathcal{Q}_{T_\theta}}\sum_{T\in\tubes(T_\theta)}|Q\cap Y^\prime(T)|\gtrapprox \sum_{T\in\tubes(T_\theta)}|Y^\prime(T)|.
$$

For each $T_\theta\in\tubes_\theta$ and each $T\in\tubes(T_\theta)$, define 
$$
Y^{\prime\prime}(T)=Y^\prime(T)\cap\bigcup_{Q\in\mathcal{Q}_{T_\theta} }Q,
$$ 
and for each $T_\theta\in\tubes_\theta$, define
$$
\tubes^\prime(T_\theta)=\{T\in\tubes(T_\theta)\colon |Y^{\prime\prime}(T)|\geq |\log\delta|^{-C_3}\lambda|T|\}.
$$
If the constant $C_3$ is chosen sufficiently large ($C_3=100$ will certainly suffice), then 
\begin{equation}\label{mostMassKept}
\sum_{T\in\tubes^\prime(T_\theta)}|Y^{\prime\prime}(T_\theta)|\gtrapprox\sum_{T\in\tubes(T_\theta)}|Y^{\prime\prime}(T_\theta)|.
\end{equation}
In particular, $\tubes^\prime(T_\theta)$ is non-empty. Note that for each $T\in\tubes_\theta$, each tube $T\in\tubes^\prime(T_\theta)$ intersects at least $|\log\delta|^{-C_3}\lambda\theta^{-1}$ cubes from $\mathcal{Q}_{T_\theta}$. 

We will now abuse notation slightly and replace each tube $T_\theta\in\tubes_\theta$ with its 10-fold dilate. In particular, if a $\theta$--cube intersects a fat tube from $\tubes_\theta$ (before the dilation is applied), then the cube is contained in the dilated version of the fat tube. We will further abuse notation and refer to these (dilated) fat tubes as ``$\theta$ tubes'' or ``fat tubes.'' As noted in Remark \ref{reasonFor20}, the dilated fat tubes are still essentially distinct. 

For each $T_\theta\in\tubes_\theta$, define $Y(T_\theta)=\bigcup_{Q\in\mathcal{Q}_{T_\theta}}Q$. Because of our dilation, this set is contained in $T_\theta$. Since $\tubes^\prime(T_\theta)$ is non-empty and since at least $|\log\delta|^{-C_3}\lambda\theta^{-1}$ cubes intersect $Y(T)$ for each $T\in\tubes^\prime(T_\theta)$, we have 
\begin{equation}\label{lambdaBdForYTTheta}
|Y(T_\theta)|\gtrapprox \theta^4 (\lambda\theta^{-1})\gtrapprox \lambda|T_\theta|. 
\end{equation}

Next, we will show that the shading $Y(T_\theta)$ satisfies the two-ends condition with exponent $\epsilon_0$ and error $\lessapprox\alpha$. The key observation is that if $T\in\tubes^\prime(\theta)$, then the shading $Y^{\prime\prime}(T)$ satisfies the two-ends condition with exponent $\epsilon_0$ and error $\lessapprox\alpha$; this is because $Y(T)$ satisfies the two-ends condition with exponent $\epsilon_0$ and error $\alpha$, and $|Y^{\prime\prime}(T)|\gtrapprox|Y(T)|$. We will use this observation on line four of the computation below. The bound \eqref{mostMassKept} will be used on line five of the computation below. Let $T_\theta\in\tubes_\theta$ and let $B(r)$ be a ball of radius $r\geq\theta$. We have
\begin{equation*}
\begin{split}
|Y(T_\theta)\cap B(r)|&\sim \theta^4 |\{Q\in\mathcal{Q}_{T_\theta}\colon Q\cap B(r)\neq\emptyset\}|\\
&\sim \theta^4 w_{T_\theta}^{-1} \sum_{T\in\tubes(T_\theta)}|Y^{\prime\prime}(T)\cap B(r)|\\
&\leq \theta^4 w_{T_\theta}^{-1} \sum_{T\in\tubes^\prime(T_\theta)}|Y^{\prime\prime}(T)\cap B(r)|\\
&\lessapprox \alpha r^{\eps_0} \theta^4 w_{T_\theta}^{-1} \sum_{T\in\tubes^\prime(T_\theta)}|Y^{\prime\prime}(T)|\\
&\lessapprox \alpha r^{\eps_0} \theta^4 w_{T_\theta}^{-1} \sum_{T\in\tubes(T_\theta)}|Y^{\prime\prime}(T)|\\
&\sim \alpha r^{\eps_0}\theta^4 |\mathcal{Q}_{T_\theta}|\\
&\sim \alpha r^{\eps_0}|Y(T_\theta)|.
\end{split}
\end{equation*}

\subsubsection{Analyzing coarse scale behavior}
After a refinement of each shading $Y(T_\theta)$ (which induces a refinement of the shading $Y^{\prime\prime}$ of the tubes in $\tubes(T_\theta)$), we can assume that there is a set $B_{\operatorname{coarse}}$ and a number $\mu_{\mathrm{coarse}}$ so that $\sum\chi_{Y(T_\theta)}\sim\mu_{\mathrm{coarse}}\chi_{B_\mathrm{coarse}}$ pointwise. After this refinement, each tube still obeys the two-ends condition with exponent $\epsilon_0$ and error $\lessapprox\alpha$. Observe that by \eqref{pointwiseMuBdInsideTube}, we have the pointwise bound
\begin{equation}\label{fineCoarseMu}
\mu \lessapprox \mu_{\mathrm{coarse}}\ \mu_{\mathrm{fine}}.
\end{equation} 

\begin{remark}
The LHS of \eqref{fineCoarseMu} might be much smaller than the RHS. Inequality \eqref{fineCoarseMu} would be sharp if at every point at which the shadings of two $\theta$--tubes intersect, the shadings of their associated $\delta$--tubes also intersect.
\end{remark}

Note that the set $B_\mathrm{coarse}$ is a union of $\theta$--cubes from $\mathcal{Q}$. For each such cube $Q$ and each $x\in Q$, the thin tubes with $x\in Y^{\prime\prime}(T)$ lie in the $\theta$--neighborhood of a plane $\Pi_x$. Thus the fat tubes whose shadings contain $x$ are contained in a union of $\theta$--neighborhoods of planes; Since all of the tubes are contained in $B(0,2)$, we can intersect these $\theta$--neighborhoods of planes with $B(0,2)$. Each such set is contained in a rectangular prism of dimensions $4\times 4\times\theta\times\theta$. We will call sets of this form ``fat planes.''

After pigeonholing, we can refine the set $B_\mathrm{coarse}$ (which induces a refinement of the shadings $Y^{\prime\prime}(T)$ and $Y(T_\theta)$) so that there is a number $B$ with the property that for each $\theta$ cube $Q$ contained in $B_\mathrm{coarse}$, we can cover the fat tubes $T_\theta$ satisfying $Q\subset Y(T_\theta)$ with $\lessapprox\max(1,\theta\mu_{\mathrm{coarse}}B)$ essentially distinct fat planes, each of which contains $\lessapprox\theta^{-1}B^{-1}$ fat tubes.  Thus we can cover the thin tubes $\{T\in\bigcup_{T_\theta\in\tubes_\theta}\tubes(T_\theta)\colon Y^{\prime\prime}(T)\cap Q\neq\emptyset\}$ with $\lessapprox\max(1,\theta\mu_{\mathrm{coarse}}B)$ essentially distinct fat planes, each of which contains $\lessapprox\theta^{-1}B^{-1}$ fat tubes. Note that if we refine the set of fat tubes or the shadings of the tubes, the above observations remain true.

We note that the $Y(T)$ passing through a given $\delta$-cube all lie in the $\theta$-neighborhood of a single 2-plane - that is, in a single fat plane.  That fat plane contains $\lessapprox \theta^{-1} B^{-1}$ fat tubes.  Therefore, 

\begin{equation} \label{muBbound}
\mu \lessapprox \theta^{-1} B^{-1} \mu_{\ofine}
\end{equation}

Plugging in our bound for $\mu_{\ofine}$ in (\ref{pointwiseMuBdInsideTube}), we get a first estimate for $\mu$:

\begin{equation} \label{muest1}
\mu \lessapprox  \lambda^{- 1/2} \delta^{-1} A^{-1/3} B^{-1} (\delta^3 | \tubes| )^{1/3}.
\end{equation}

We will get a complementary estimate by studying $\mu_{\ocoarse}$.  Let $T_{\theta}$ be a fat tube. Then the fat tubes passing through each $\theta$--cube in $Y(T_{\theta})$ are contained in a union of fat planes. The same fat plane might be associated to several different $\theta$ cubes contained in $Y(T_{\theta})$. Refining the shadings $Y(T_\theta)$ and the set of fat tubes $\tubes_\theta$, there is a number $D$ so that for each fat plane containing $T_{\theta}$, there are $\sim D$ cubes contained in $Y(T_\theta)$ that are associated to that fat plane. In particular, this means that the tubes in the set
$$
\operatorname{Hair}(T_\theta)=\{T_\theta^\prime\in\tubes_{\theta}\colon Y(T_{\theta}^\prime)\cap Y(T_\theta)\neq\emptyset\}
$$ 
are contained in a union of fat planes of volume $\lessapprox \theta^2 \mu_{\ocoarse} D^{-1}B$, and this property is preserved under refinements. Furthermore, for each plane $\Pi$ with $T_\theta\subset N_{\theta}(\Pi)$, at most $\lesssim D$ tubes in $\tubes_{T_\theta,\Pi}=\{T_\theta^\prime\in \operatorname{Hair}(T_\theta)\colon T_\theta^\prime\subset N_{\theta}(\Pi)\}$ can point in the same direction. This is because the tubes in $\tubes_{T_\theta,\Pi}$ intersect $T_{\theta}$ in $\lesssim D$ distinct cubes, and the tubes in $\tubes_{T_\theta,\Pi}$ are essentially distinct. This means that for each $\theta$--separated direction $v$, at most 1 tube from $\tubes_{T_\theta,\Pi}$ can point in direction $v$ for each of the $\lesssim D$ distinct $\theta$--cubes where the intersections between the tubes in $\tubes_{T_\theta,\Pi}$ and $T_\theta$ occur. 

 Applying Proposition \ref{smallHairbrushProp} with 
\begin{equation*}
\rho\lesssim \theta^2 \mu_{\mathrm{coarse}} D^{-1}B,\qquad E\lessapprox A(\delta^3|\tubes|)^{-1},\qquad |\tubes_\theta|=A\theta^{-3},
\end{equation*}
we conclude that
\begin{equation*}
\begin{split}
\mu_{\mathrm{coarse}} &\lessapprox_{\eps_0,\alpha} \lambda^{-1} \theta^{-1}\qquad \ E^{1/2}\qquad\quad \phantom{.}D^{1/4}\qquad\ \ \rho^{1/4}\qquad\qquad\ \ \ \big(\theta^3|\tubes_\theta|\big)^{1/4}\\
 &\lessapprox_{\eps_0,\alpha}  \lambda^{-1} \theta^{-1}\ \big(A(\delta^3|\tubes|)^{-1}\big)^{1/2}\ D^{1/4}\ \big(\theta^2\mu_{\mathrm{coarse}}D^{-1}B\big)^{1/4}\ \big(\theta^3(A\theta^{-3})\big)^{1/4},
\end{split}
 \end{equation*}
i.e.
 \begin{equation*}
\mu_{\mathrm{coarse}}^{3/4}\lessapprox_{\eps_0,\alpha} \lambda^{-1} \theta^{-1/2}A^{3/4}B^{1/4}(\delta^3|\tubes|)^{-1/2},
 \end{equation*}
so
\begin{equation}\label{estimate2}
 \mu_{\mathrm{coarse}} \lessapprox_{\eps_0,\alpha} \lambda^{-4/3} \theta^{-2/3}AB^{1/3}(\delta^3|\tubes|)^{-2/3}.
\end{equation}

Combining the bound for $\mu_{\ofine}$ in (\ref{pointwiseMuBdInsideTube}) and this bound for $\mu_{\ocoarse}$, we get a second bound for $\mu \lessapprox \mu_{\ofine} \mu_{\ocoarse}$:

\begin{equation} \label{muest2}
\mu \lessapprox_{\eps_0,\alpha} 
\lambda^{- 11/6} \theta^{1/3} \delta^{-1} A^{2/3} B^{1/3} (\delta^3 | \tubes| )^{-1/3}.
\end{equation}

Combining (\ref{muest1}) and (\ref{muest2}), we get the following bound for $\mu$:

\begin{equation*}
\begin{split}
\mu & \lessapprox_{\eps_0,\alpha} \big(\lambda^{- 1/2} \delta^{-1} A^{-1/3} B^{-1} (\delta^3 | \tubes| )^{1/3}\big)^{2/3}\big(\lambda^{- 11/6} \theta^{1/3} \delta^{-1} A^{2/3} B^{1/3} (\delta^3 | \tubes| )^{-1/3}\big)^{1/3}\\
& = \lambda^{-17/18} \theta^{1/9} \delta^{-1} B^{-5/9}(\delta^3 | \tubes| )^{1/9}.
\end{split}
\end{equation*}

Since know $B \ge 1$, we get the simpler bound

$$ \mu \lessapprox \lambda^{-17/18} \theta^{1/9} \delta^{-1}(\delta^3 | \tubes| )^{1/9}. $$

Recall that $\sum_{T \in \tubes} \chi_{Y'(T)} \sim \mu \chi_W$.  

$$|W| \mu \approx \lambda \delta^3 | \tubes |, $$

and so we get a lower bound for $|W|$,

$$ |W| \gtrapprox_{\eps_0,\alpha} \lambda^{35/18} \theta^{-1/9} \delta ( \delta^3 | \tubes |)^{8/9}. $$

 This concludes the proof of Proposition \ref{badKPlainyBound} and hence of Theorem \ref{plainyTubesBound}.
\section{Proof of Theorem \ref{mainThm}}\label{proofOfMainThmSection}
Let $(\tubes,Y)$ satisfy the hypotheses of Proposition \ref{mainPropQuantTrans}, and let
\begin{equation*}
\theta_0=(\delta/\lambda)^{9/40}.
\end{equation*}
Observe that $\theta_0<1$, since we can assume $|\tubes|\geq 1$. Let $X_1$ be the set of points $x\in \RR^4$ for which there exists plane $\Pi$ passing through $x$ with the property that
$$
|\{T\in\tubes\colon x\in Y(T), \angle(T,\Pi)<\theta_0\}|\geq \frac{1}{100} |\{T\in\tubes\colon x\in T\}|,
$$
and let $X_2=\RR^4\backslash X_1$. Since $\sum_{T\in\tubes}|Y(T)|\geq (\lambda/2) (\delta^3|\tubes|)$, at least one of the following must hold:

\begin{align}
\sum_{T\in\tubes}|Y(T)\cap X_1|&\geq (\lambda/4)(\delta^3|\tubes|)\qquad\textrm{(the tubes are $\theta$--plainy)},\label{plainynessHolds}\\
\sum_{T\in\tubes}|Y(T)\cap X_2|&\geq (\lambda/4)(\delta^3|\tubes|)\qquad\textrm{(the tubes are $\theta$--trilinear)}.\label{plainynessFails}
\end{align}
Suppose \eqref{plainynessHolds} holds. Refine the shadings $Y(T)$ so that for each $x\in X_1$, the tubes $\{T\in\tubes\colon x\in Y(T)\}$ are contained in the $\theta_0$--neighborhood of a plane. After this refinement we have $\sum_{T\in\tubes}|Y(T)|\geq \frac{\lambda}{400}(\delta^3|\tubes|)$. Thus we can find a set $\tubes^\prime\subset\tubes$ with $|\tubes^\prime|\geq \frac{1}{400}|\tubes|$, and after further refining $Y(T)$ by a factor of at most $400$, we have $\frac{\lambda}{400}\leq|Y(T)|/|T|\leq \frac{\lambda}{200}$ for each $T\in\tubes^\prime$. In particular, this implies that the tubes in $\tubes^\prime$ satisfy the two-ends condition with exponent $\epsilon_0$ and error $400\alpha$. Applying Theorem \ref{plainyTubesBound} to $\tubes^\prime$, we conclude that
\begin{equation}\label{planynessBd}
\begin{split}
\Big|\bigcup_{T\in\tubes}Y(T)\Big|&\gtrapprox c_{\epsilon,\epsilon_0} \alpha^{-C^\prime_{\epsilon,\epsilon_0}} \lambda^{3} K^{-1} \theta_0^{-1/9}\delta^{1+\epsilon}(\delta^3|\tubes|)\\
&=c_{\epsilon,\epsilon_0} \alpha^{-C^\prime_{\epsilon,\epsilon_0}} \lambda^{3+\gain} K^{-1}\delta^{1-\gain+\epsilon}(\delta^3|\tubes|).
\end{split}
\end{equation}

Now suppose that \eqref{plainynessFails} holds. Replace the shading $Y(T)$ by $Y(T)\cap X_2$. Again, we can find a set $\tubes^\prime\subset\tubes$ with $|\tubes^\prime|\geq\frac{1}{4}|\tubes|$ and after refining $Y(T)$ by a factor of at most 8, we have $\frac{\lambda}{4}\leq |Y(T)|/|T|\leq\frac{\lambda}{2}$ for each $T\in\tubes$. Apply Corollary\ref{volumeTrilinearCor} (using the bound \eqref{alternateTrilinBound}) to $|\tubes^\prime|$. We have
\begin{equation}\label{trilinBound}
\Big|\bigcup_{T\in\tubes}Y(T)\Big|\gtrapprox \lambda^{3+1/4}K^{-1} \theta_0 \delta^{3/4}(\delta^3|\tubes|)=\lambda^{3+\gain}K^{-1} \delta^{1-\gain}(\delta^3|\tubes|).
\end{equation}

Since at least one of \eqref{planynessBd} or \eqref{trilinBound} must hold, we obtain Proposition \ref{mainPropQuantTrans}. This concludes the proof of Theorem \ref{mainThm}.

\section{Minimal conditions for Theorem \ref{mainThm}} \label{minimalConditions}
An examination of the proof of Proposition \ref{techVersion3lin} reveals that the polynomial Wolff axioms are only used in one place (Equation \eqref{boundOnN1}), and as discussed in Remark \ref{whereWolffAxiomsAreUsed}, \eqref{boundOnN1} holds provided the tubes satisfy a restricted version of the polynomial Wolff axioms. Indeed, we have the following variant of Proposition \ref{techVersion3lin}.

\begin{prop}\label{easierToSatisfyProp}
For all $\epsilon>0$, there exist constants $C_\epsilon, d(\epsilon)$ so that the following holds. Let $\tubes$ be a set of $\delta$ tubes in $\RR^4$. Suppose that for every integer $1\leq E\leq d(\epsilon),$ for every polynomial $P\in\RR[x_1,x_2,x_3,x_4]$ of degree $E$, for every ball $B(x,r)$ of radius $r$, and for every $w>0$, we have
\begin{equation}\label{fewTubesInHypersurfaces}
|\{T\in\tubes\colon T\cap B(x,r)\subset N_{10\delta}(Z)\}|\leq \tilde K_{E,w}\ r^{-1}\delta^{-2-w}.
\end{equation}
Then 
\begin{equation}\label{trilinearBoundFixedP}
\int\Big(\sum_{T_1,T_2,T_3\in\tubes}\chi_{T_1}\ \chi_{T_2}\ \chi_{T_3}\ |v_1\wedge v_2\wedge v_3|^{12/13}\Big)^{13/27} \leq C_{\epsilon}\delta^{-1/3-\epsilon}\tilde K^{1/9}(\delta^3|\tubes|)^{4/3},
\end{equation}
where in the above expression $v_i$ is the direction of the tube $T_i$, and $\tilde K=\tilde K_{\tubes, d(\epsilon)}=\sup_{1\leq E\leq d(\epsilon)}{\tilde K_E}$, where $\{\tilde K_E\}$ are the constants from \eqref{fewTubesInHypersurfaces}. 
\end{prop}

If we use Proposition \ref{easierToSatisfyProp} in place of Proposition \ref{techVersion3lin}, we obtain the following variant of Proposition \ref{mainPropQuantTrans}.

\begin{prop}
For all $\epsilon>0$, $\epsilon_0>0$ there exist constants $c_{\epsilon,\epsilon_0}>0, C^\prime_{\epsilon,\epsilon_0}$, and $d(\epsilon)$ so that the following holds. Let $\tubes$ be a set of $\delta$--tubes in $\RR^4$.  Suppose that $\tubes$ satisfies the linear Wolff axioms: for every rectangular prism $R$ of dimensions $1\times t_1\times t_2\times t_3$, we have that  $\lesssim t_1t_2t_3\delta^{-3}$ tubes from $\tubes$ can be contained in $R$. Suppose furthermore that for every integer $1\leq E\leq d(\epsilon),$ for every polynomial $P\in\RR[x_1,x_2,x_3,x_4]$ of degree $E$, for every ball $B(x,r)$ of radius $r$, and for every $w>0$, we have
$$
|\{T\in\tubes\colon T\cap B(x,r)\subset N_{10\delta}(Z)\}|\leq \tilde K_{E,w} \ r^{-1}\delta^{-2-w}.
$$

For each $T\in\tubes$, let $Y(T)\subset T$ with $\lambda\leq |Y(T)|/|T|\leq 2\lambda$. Suppose that $(\tubes,Y)$ is $s$--robustly transverse (with error $1/100$) and that each tube $T\in\tubes$ satisfies the two-ends condition with exponent $\epsilon_0$ and error $\alpha$. Then 
\begin{equation} \label{volumeBoundExplicitTrans}
\Big| \bigcup_{T\in\tubes}Y(T)\Big| \geq c_{s}c^\prime_{\epsilon,\epsilon_0}\alpha^{-C^\prime_{\epsilon,\epsilon_0}}\lambda^{3+\gain}\tilde K^{-1} \delta^{1-\gain+\epsilon}\big(\delta^3|\tubes|\big),
\end{equation}
where $\tilde K=\tilde K_{d(\epsilon),\tubes}=\sup_{1\leq E\leq d(\epsilon)}{\tilde K_E}$.
\end{prop}

Recall that Conjecture \ref{polyWolffConj} would imply a Kakeya maximal function estimate at dimension $3+\gain$. Proposition \ref{easierToSatisfyProp} says that an easier variant of Conjecture \ref{polyWolffConj} would also imply a Kakeya maximal function estimate at dimension $3+\gain$; rather than proving the full strength of Conjecture \ref{polyWolffConj}, it suffices to show that if $\tubes$ is a set of $\delta$ tubes in $\RR^4$ that point in $\delta$-separated directions, then for each integer $E\geq 1$, \eqref{fewTubesInHypersurfaces} holds for every polynomial $P\in\RR[x_1,x_2,x_3,x_4]$ of degree $E$.

In $\mathbb{F}_p^4$, the analogous statement would be that if $\mathcal{L}$ is a set of $p^3$ lines pointing in different directions, then any degree $D$ hypersurface contains $O_D(p^2)$ lines from $\mathcal{L}$. This is easy to prove: simply embed $\mathbb{F}_{p}^4$  into four-dimensional projective space and let $\hat Z$ be the corresponding hypersurface. If $\Pi$ is the hyperplane at infinity, then each line from $\mathcal{L}$ intersects $\Pi$ at a distinct point. Thus $|\mathcal{L}|\leq |\hat Z \cap \Pi|\leq Dp^2$. See \cite{T} for details. Unfortunately, it appears to be difficult to make a similar argument work in Euclidean space. We believe that this is a promising direction for future study.

\bibliographystyle{abbrv}
\bibliography{R4_Kakeya}

 \bigskip
  \footnotesize

\noindent L.~Guth, Department of Mathematics, Massachusetts Institute of Technology,
  77 Massachusetts Avenue, Cambridge, MA 02139-4307, USA.

  \medskip

\noindent  J.~Zahl, Department of Mathematics, University of British Columbia,
1984 Mathematics Road, Vancouver, BC, V6T 1Z2, Canada.

\medskip

\noindent MSC2010 classification: 42B25

\end{document}